\documentclass[11pt]{article}
\usepackage{mathptmx, amsfonts,amssymb,amsmath,amsfonts,amsthm, mathrsfs}
\usepackage{bbm}

\usepackage{color}
\usepackage[colorlinks=true,allcolors=blue]{hyperref}
\usepackage[nameinlink,noabbrev,capitalize]{cleveref}
\def\one{\mbox{1\hspace{-4.25pt}\fontsize{12}{14.4}\selectfont\textrm{1}}} % 11pt

\usepackage{graphicx}
\usepackage{caption}

\usepackage{geometry}
\newenvironment{dedication}
{\vspace{.1cm}\begin{quotation}\begin{center}\begin{em}}
			{\par\end{em}\end{center}\end{quotation}\vspace{1cm}}

\usepackage{bm}

\DeclareMathOperator{\E}{\mathbb{E}}

\usepackage[utf8]{inputenc}
\usepackage[english]{babel}

\newtheorem{theorem}{Theorem}[section]
\newtheorem{corollary}{Corollary}[theorem]

\newtheorem{definition}{\noindent \bf{Definition}}[section]

\newtheorem{remark}{\noindent \bf{Remark}}[section]

\newtheorem{asmpt}{\noindent \bf{Condition}}[section]

\numberwithin{equation}{section}

\begin{document}
	\title{Two queues with random time-limited polling}
	\author{Mayank Saxena\thanks{Department of Mathematics and Computer Science, Eindhoven
			University of Technology,  The
			Netherlands}\, \thanks{Email: m.mayank@tue.nl}
		\and Onno Boxma$^*$\thanks{Email: o.j.boxma@tue.nl} \and Stella Kapodistria$^*$\thanks{Email: s.kapodistria@tue.nl}
		\and Rudesindo N\'u\~nez Queija \thanks{Korteweg-de Vries Institute for Mathematics, University of Amsterdam, The Netherlands, Email: nunezqueija@uva.nl}\,	\thanks{CWI, Amsterdam, The Netherlands}	}

	\maketitle
	
%%******************************************* Abstract ***************************************%%
\begin{dedication}This paper is dedicated to Tomasz Rolski, in friendship, 
		respect and admiration. His love of applied probability and 
		never-ending curiosity are a blessing for our field.
\end{dedication}

\noindent
\textit{Keywords}: {polling model; workload decomposition; heavy traffic; heavy tail asymptotics; singular perturbation analysis; time-scale separation;
	geometric ergodicity}                                   %key words and phrases

%\MSCcodes{68M20;90B22.}               %2000 Mathematics Subject Classification

\begin{abstract}
In this paper, we analyse a single server polling model with two queues. Customers arrive at the two queues according to two independent Poisson processes. There is a single server that serves both queues with generally distributed service times. The server spends an exponentially distributed amount of time in each queue. After the completion of this residing time, the server instantaneously switches to the other queue, i.e., there is no switch-over time. For this polling model we derive the steady-state marginal workload distribution, as well as heavy traffic and heavy tail asymptotic results. Furthermore, we also calculate the joint queue length distribution for the special case of exponentially distributed service times using singular perturbation analysis.
\end{abstract}
\section{Introduction}
In this paper, we are interested in the performance analysis of a single server polling model with a special service discipline (i.e., the criterion which determines how many customers are served during a visit of the server to a queue).
A typical  polling model consists of multiple queues, attended by a single server who visits the queues in some order to render service to the customers waiting
at the queues. Moving from one queue to another, the server incurs a (possibly zero) switch-over time. Once the server is at one of the queues, the server serves the customers of that queue based on a service discipline and according to some service time distribution.

Polling models were initially introduced in the 1950's but mostly gained their popularity during the 1990's. This popularity rise was due to the wide range of applicability of polling models, especially for the modelling of  computer-communica-tion systems and protocols, traffic signal management, and manufacturing, see, e.g., \cite{Takagi-book,Takagi-00,vishnevskii2006mathematical} for a series of comprehensive surveys and \cite{Boon-Mei-Winards,levy1991polling,Takagi-91} for extensive overviews of the applicability of polling systems.

The performance analysis of polling models has received considerable attention, see, e.g., \cite{takagi1988queuing}.
In particular, in the polling literature much attention has been given to determining the probability generating function (PGF) of the joint queue length distribution under stationarity and at various epochs.
A wide range of service disciplines has been considered, including {\em exhaustive service}
(per visit to a queue, the server continues to serve all customers until it empties)
and {\em gated service} (per visit to a queue, the server serves only those customers
which are already present at the start of the visit).
In \cite{Resing}, Resing shows that the joint queue length PGF of polling models
in which the service discipline satisfies the so-called branching property
equals the (known) PGF of
a multi-type branching process with immigration. Service disciplines which satisfy the branching property include
the exhaustive and gated disciplines. Polling systems with disciplines which do not satisfy the branching property usually defy an exact analysis. In our paper, we assume that the server spends an exponentially distributed amount of time at each queue. Upon the completion of this residing time at each queue, the server instantaneously switches to another queue according to a cyclic order. Such a service protocol does not exhibit the branching property, which complicates the analysis significantly. We concentrate on the two-queue model and, whenever possible, suggest extensions to the multi-queue model. A similar service discipline has been considered in \cite{EliazarYechaili}, \cite{AlHanbali}, \cite{xie1997workload}, \cite{de2009polling}, and the references therein.

\paragraph{Related literature}
In \cite{EliazarYechaili}, Eliazar and Yechiali consider a multi-queue polling system under the Randomly Timed Gated (RTG) service discipline. The RTG discipline operates as follows: whenever the server enters a station, a timer is activated. If the server empties the queue before the timer's expiration, the server moves on to the next queue. Otherwise (i.e., if there is still work in the station when the timer expires), the server obeys one of the following rules, each leading to a different model: (1) The server completes all the work accumulated up to the timer's expiration and then moves on to the next node. (2) The server completes only the service of the job currently being served, and moves on. (3) The server stops working immediately and moves on. The model suggested in this manuscript bears resemblance to rule  (3), however, in our case if a queue becomes empty, the server does not switch, and only does so when the timer expires. Eliazar and Yechiali, in \cite{EliazarYechaili}, produce a recursive expression for the PGF of the number of customers in the queues of the polling model, while the case of two queues is sketched in \cite{Coffmanetal}, by a transformation to a boundary value problem. In \cite{Katayama2001}, Katayama using a level-crossing approach, obtains the Laplace-Stieltjes transforms (LSTs) and the moment formulas for the waiting times and the sojourn times, and based on these expressions, he also proves a decomposition property.

In \cite{AlHanbali}, the authors consider a polling model with Poisson batch arrivals and phase-type service times, and an exponential service timer. The authors establish a relation for the PGF of the number of customers in the queue at the beginning and at the end of the server's visit to a queue.  This is used as an input for a numerical scheme that is used to approximate the joint queue-length distribution at the server departure instants from the queues.

In \cite{xie1997workload}, Xie et al.\ consider a single server multi-queue system, in which the server visits the individual queues for a fixed amount of time in a deterministic, cyclic order. They refer to the fixed residing time as the {\em orientation} time. They argue that such a service discipline comes with two operational advantages: it enables to i) keep the frequency of switching at a predetermined level (thus controlling the total cost, if there is a switching cost), ii) balance the time that the server spends in each queue (since, contrary to exhaustive or gated service disciplines, this discipline does not depend on the number of customers present in the various queues).

In \cite{de2009polling}, the authors assume a random visit (residing) time for each queue, which is independent of the number of customers present at each queue, and a {\em preemptive-repeat with resampling} service strategy. This autonomous service discipline is motivated from wireless ad hoc networks with movable communication hops.
Another application is in single upstream tree-based ethernet passive optical networks, in which the central optical line terminal dedicates the channel to a specific user (e.g., the user with the highest priority) for a random amount of time, see \cite{kramer2002ipact} and the references therein. For more applications on this type of autonomous service disciplines, the interested reader is referred to \cite{altman1994analysing}.
For all aforementioned applications, we consider it natural to assume that the service strategy is {\em preemptive-resume} and that the switch-over time is negligible in comparison to the service time and the residing time.

\paragraph{Paper originality}
In this paper, we initially devote attention to the individual queues.
When focusing on a single queue, the model can be interpreted as a service system with {\em vacations}: we interpret the time that the server visits the other queue as a vacation period.
Vacation queues - and {\em priority} queues for which the mathematical analysis is similar - are well studied in the queueing literature starting with the work of White and Christie~\cite{whitechristie58} (exponentially distributed service times and vacations), Gaver\cite{gaver62}, Thiruvengadam~\cite{thiruvengadam63} and Avi-Itzhak and Naor~\cite{avi-naor63} (the latter three assuming generally distributed service times and vacations).
All these works assume that the service periods have an exponential distribution, but vary for example in the assumptions regarding whether interrupted services are resumed or repeated and in the metrics of interest.
Takagi~\cite{Takagi-book} provides an excellent overview of vacation and priority systems.
The interested reader is also referred to Federgruen and Green~\cite{feder-green86} for phase-type distributed service periods, to Takine and Sengupta~\cite{takinesengupta97} for Markovian arrival processes and to Fiems et al.~\cite{fiems-maer-brun08} for a more recent publication with various sorts of service disruptions.
For a more extensive overview of the literature, we refer to the recent survey of Krishnamoorthy et al.~\cite{krishna-etal14}.

A particular feature of a large class of vacation queues is that the stationary workload and queue length distributions obey a stochastic decomposition property, as first observed by Gaver~\cite{gaver62} and Miller~\cite{miller64}.
Fuhrmann and Cooper~\cite{fuhrmancooper85} give conditions for such a queue length decomposition to hold.
Our model does not satisfy these conditions, but we show that it does allow a stochastic
decomposition of the stationary workload. It is notable that the workload can be decomposed into two independent terms:  the amount of work of an M/G/1 queue, and the amount of work when the server is not serving the first queue, due to either an idle period or due to a visit at the second queue. The second term in the decomposition is at first sight surprising, as one would typically expect that the decomposition is identical with that of the standard M/G/1 queue with exhaustive service and multiple vacations, as is the case for the RTG, see \cite[Remarks 2 and 4]{Katayama2001}.
%For these reasons, the second term of the decomposition could not have been intuitively predicted without the exact analysis of the marginal workload distribution.
In that respect, the exact analysis of the marginal workload distribution was quite helpful to us.
The use of the decomposition property further facilitates our heavy traffic and heavy tail analyses, as we can use known results for the M/G/1 queue and  restrict our attention to the second term appearing in the stochastic decomposition.

In this paper, we also devote attention to the joint workload distribution. Using similar probabilistic arguments as in the analysis of the marginal workload, we show that the joint LST of the workloads at the queues satisfies a functional equation \eqref{eq6.4}, which is, in the case of identical queues, then reduced to a Dirichlet boundary value problem. Thus, one can numerically evaluate the mapping from the contour defined by the kernel of the functional equation to the unit circle, and obtain a solution to the joint distribution.
However, we have to note that, depending on the service time distribution, the kernel of the functional equation does not have the typical (quadratic) polynomial form, which complicates tremendously the analysis and differentiates it from the known results of the literature, see, e.g., \cite{cohen2000boundary,Coffmanetal,fayolle1999}. In the case of exponentially distributed service times, we propose perturbation analysis for the calculation of the joint queue length distribution, as a methodological alternative to the boundary value problem approach. By appropriately scaling the arrival and service rates by a factor $\varepsilon$, the invariant probability measure of the perturbed Markov chain is written as a
power series expression in terms of $\varepsilon$, whose coefficients form a geometric sequence, that can be used for both exact and numerical calculations. Furthermore, we show that there exists a computationally stable updating formula for the calculation of the perturbed invariant measure.
To approximate the joint distribution numerically, one needs to solve a large system of equations, for which we indicate two possible
approaches, but we do not pursue these in this paper.

\paragraph{Paper overview}
The paper is organized as follows: In Section \ref{ModelDescription}, we describe the two-dimensional polling model under  consideration. In Section \ref{sec:marginal}, we present the LSTs of the model's {\em marginal} workload distributions in steady-state at an arbitrary epoch. In Section \ref{sec:WorkloadDec}, we show that a single queue's marginal workload satisfies a decomposition property and then by using the decomposition property in
the light tailed case we obtain the heavy traffic limit of the marginal workload distributions in steady-state. In Section \ref{sec:Asympt}, we discuss the heavy tail asymptotics of the marginal workload distributions in steady-state, and in that case we also discuss the heavy traffic behavior. We then discuss, in Section \ref{sec:JointWork}, open problems arising in the calculation of the {\em joint} workload distribution.
Assuming exponentially distributed service times, we calculate in Section \ref{sec:JointQueue} the joint queue length distribution in  steady-state at an arbitrary epoch using singular perturbation analysis. Several possible future research directions are discussed in Section \ref{sec:FutureDir}.

%==================================== Section 1: Notation and Model Description ====================================

\section{Model description and notation}\label{ModelDescription}
In this paper, we consider a two-queue polling model. Customers arrive to queue $i$ according to a Poisson process at rate $\lambda_i$, $i=1,2$. There is a single server, that serves both queues according to the first come first serve (FCFS)  discipline. The service times of customers in queue $i$ are independent and identically generally distributed positive random variables, say $B_i$, $i=1, 2$. We denote the LST of the service time $B_i$ by $\tilde{b}_i(s)=\mathbb{E}(e^{-s B_i})$, with $\mathrm{Re}~s\geq0$, $i = 1, 2$.

A special feature of the polling model under consideration is that the server spends an exponentially distributed amount of time at queue $i$ with rate $c_i$, $i=1, 2$. Upon completion of the residing time at queue $i$, the server instantaneously switches to the other queue, i.e., there is no switch-over time. Furthermore, if upon completion of the residing time, the server is providing service to a customer, this service is interrupted and resumed at the next visit of the server to the queue. More explicitly, we assume that if a server resumes the service after being interrupted, the server continues from where the service stopped instead of starting from the beginning, i.e., the service is \textit{preemptive--resume}. We denote the LST of the residing time $T_i$ of the server in queue $i$ by $\tilde{f}_{T_i}(s)=\mathbb{E}(e^{-sT_i})$, with $\mathrm{Re}~s\geq0$, $i = 1, 2$, with $T_i$ exponentially distributed with rate $c_i$ and probability density function $f_{T_i}(t)= c_i e^{-c_i t}$, $t\geq0$, $i=1,2$.

\paragraph{Stability condition} For the two-queue polling model under consideration the stability condition (sufficient and necessary) is
\begin{equation}\label{Stability}
\rho_1< \frac{c_2}{c_1 + c_2} ~ \text{ and }~  \rho_2< \frac{c_1}{c_1 + c_2},
\end{equation}
with $\rho_i=\lambda_i\, \mathbb{E}(B_i)$, $i=1,2$.

The stability condition can be proven by appropriately adapting and extending the proof of Altman et al. in \cite{Altmanetal}. To this purpose, one would need to calculate the expected increase of the workload during a cycle  (i.e., the time between two successive arrivals of the server at the first queue) and use an extension of Foster's criterion known as {\em the positivity/regularity criterion (V2)}, cf. \cite{Meyn}. The derivations of the expected increase of the workload during a cycle are similar to the analysis performed for the proof of Theorem \ref{theorem-marginal}. This guarantees that the stability condition is sufficient. In order to show that it is also necessary, one may use the expression for the expectation of the steady-state workload of each queue, cf. Equation \eqref{Eq:Mean}.

\begin{remark} Equivalently, one can prove the stability condition by adapting
the steps presented in \cite{Cao}. More concretely, the two-queue polling system under consideration is said to be stable if the workloads at each queue, at the polling instants (i.e., the instant when the server arrives at a queue), have a proper limiting distribution, and the mean cycle time is finite, as time tends to infinity. By definition, the latter always holds, as the cycle time for the polling model under consideration is given by the sum of the two exponentially distributed residing times at the two queues. Following  a similar argumentation as in \cite{Cao}, one would need to ensure that, under the above stability condition, each queue is stable. Note that each queue in isolation behaves like an M/G/1 queue with a service speed governed by a two-state Markov chain. The stability condition (in the sense that the workload has a proper limiting distribution) of such a queueing system is studied in \cite{Baccelli}.

\end{remark}

\begin{remark} Intuitively, the stability condition for the first queue can be interpreted as follows: the long-run proportion of time the server spends in the first queue is equal to $c_2/(c_1 + c_2)$, thus the long-run rate of service in the first queue is $c_2/(c_1 + c_2)\mathbb{E}(B_1)$. Hence, for the first queue to be stable it is needed that the arrival rate  is strictly smaller than the long-run rate of service, which corresponds to the left hand side of \eqref{Stability}. The stability condition for the second queue can be interpreted in an analogous manner.
\end{remark}
%============================ Section 2: Marginal workload analysis ================================%	

\section{Marginal workload analysis}
\label{sec:marginal}
In this section, we derive the distribution of the marginal workload in steady-state at an arbitrary epoch.
As discussed in the introductory section, the individual queues behave as vacation systems: from the perspective of one queue, the server is on vacation when it resides at the other queue.
In this section, we give a direct derivation of the stationary marginal workload distributions.
\\
We let $V_i(t)$ denote the workload at time $t$, $t\geq0$, of queue $i$, $i=1,2$, and $V_i$ denote the steady-state workload of queue $i$ at an arbitrary epoch, $i=1,2$.

\begin{theorem}
\label{theorem-marginal}
	The LST of the workload of the first queue in steady-state under the stability condition \eqref{Stability} is given by
	\begin{equation}\label{MLQ1}
	\mathbb{E}(e^{-sV_1})=\frac{s\left[\lambda_1 \mathbb{E}(B_1) \left(c_1+c_2\right)-c_2\right]\left[c_1+c_2+\lambda_1\left(1-\tilde{b}_1(s)\right)\right]}{\left[\left(c_2+\lambda_1\left(1-\tilde{b}_1(s)\right)\right)\left(c_1+\lambda_1\left(1-\tilde{b}_1(s)\right)-s\right)-c_1c_2\right]\left(c_1+c_2\right)}.
	\end{equation}
	A symmetric formula holds for the LST of $V_2$ under the stability condition \eqref{Stability}.
\end{theorem}

\begin{proof} The derivation of the LST of the steady-state workload for the first queue is performed by considering the renewal process at the instances the server arrives at the first queue, i.e., the inter-renewal times are identical in distribution to $T_1+T_2$, with $T_i\sim \mathrm{Exp}(c_i)$, $i=1,2$. \\
	To structure the exposition, the proof of the theorem is split into five steps. A key point of the proof is the derivation of $\mathbb{E}(e^{-sV_1(T_1 + T_2)})$; this is achieved in Step 4, after we derive $\mathbb{E}(e^{-sV_1(T_1+T_2)}|V_1(T_1)=y)$ in Step 1, $\mathbb{E}(e^{-sV_1(T_1)}|V_1(0)=v)$ in Step 2, and subsequently $\mathbb{E}(e^{-sV_1(T_1+T_2)}|V_1(0)=v)$ in Step 3. Finally, in Step 5, we calculate $\mathbb{E}(e^{-sV_1})$ using the PASTA property and the result of Step 4.
	\begin{description}	
		\item{\bf Step 1:} Calculation of $\mathbb{E}(e^{-sV_1(T_1+T_2)}|V_1(T_1)=y)$.\\
		During $(T_1,T_1+T_2]$ the server serves only customers in the second queue, so the workload in the first queue only increases by the sum of the service times of all the customers that arrived within this interval.
		The increments occur according to a compound Poisson process. So,	
		\begin{equation}\label{Eq:firstterm}
		\begin{split}		
		\mathbb{E}(e^{-sV_1(T_1+T_2)}|V_1(T_1)=y) &= e^{-sy}\tilde{f}_{T_2}(\lambda_1(1-\tilde{b}_1(s)))
		\\&=e^{-sy}\frac{c_2}{c_2+\lambda_1(1-\tilde{b}_1(s))}.
		\end{split}
		\end{equation}

		\item{\bf Step 2:} Calculation of $\mathbb{E}(e^{-sV_1(T_1)}|V_1(0)=v)$.\\
		Note that
		\begin{align}\label{EVT1}
		\mathbb{E}({e}^{-sV_1(T_1)}|V_1(0)=v) =& \int_{t=0}^{\infty} c_1 {e}^{-c_1t}
		\int_{\sigma=0}^{\infty}e^{-s\sigma}\mathrm{d}\mathbb{P}(V_1(t)<\sigma|V_1(0)=v) {\rm d}t.
		\end{align}
		In order to calculate the right hand side of \eqref{EVT1}, we use \cite[p.~262,~Equation~(4.99)]{cohen2012single}
		\begin{align*}
		\int_{\sigma=0}^{\infty}&e^{-s\sigma}\mathrm{d}\mathbb{P}(V_1(t)<\sigma|V_1(0)=v) =  e^{s(t-v)-t\lambda_1(1-\tilde{b}_1(s))} \nonumber\\
		& -s U_1(t-v) \int_{u=0}^{t-v}e^{(s-\lambda_1(1-\tilde{b}_1(s)))(t-u-v)}\mathbb{P}(V_1(u+v)=0|V_1(0)=v)\mathrm{d}u,%\tag*{Eq.~(4.99) \cite{cohen2012single}}		
		\end{align*}	
		with $\mathrm{Re}~s \geq 0$, $t \geq 0$, and $U_1(x) = 0$, if $x < 0$, and $U_1(x) = 1$, otherwise.
		Hence, Equation \eqref{EVT1} in light of \cite[p.~262,~Equation~(4.99)]{cohen2012single} yields
		\begin{equation}
		\begin{split}		
		&\mathbb{E}({e}^{-sV_1(T_1)}|V_1(0)=v) \\& = \frac{c_1 {e}^{-sv}}{c_1 + \lambda_1(1 - \tilde{b}_1(s)) -s}		
		\\&-
		\int_{t=v}^{\infty}s
		c_1 e^{-c_1t}
		\int_{u=0}^{t - v}e^{(s - \lambda_1(1 - \tilde{b}_1(s)))(t - u - v)}\mathbb{P}(V_1(u + v) = 0|V_1(0) = v)
		\mathrm{d}u \, \mathrm{d}t .
		\label{eq3.4}
		\end{split}
		\end{equation}
		For the calculation of the integrals in the right hand side of Equation \eqref{eq3.4} we use \cite[p.~260, ~Equation~(4.92)]{cohen2012single}, for $\mathrm{Re}~s \geq 0$, $t \geq 0$,
		\begin{equation*}
		\int_{t=0}^{\infty}e^{-s t}\mathbb{P}(V_1(t)=0|V_1(0)=v)\mathrm{d}t=\frac{e^{-(s+(1-\mu(s,1))\lambda_1)v}}{s+(1-\mu(s,1))\lambda_1},%\tag*{Eq.~(4.92) \cite{cohen2012single}}
		\end{equation*}
		with $\mu(s,1)$ denoting the LST of the busy period distribution of the M/G/1 queue with arrival rate $\lambda_1$ and service time LST $\tilde{b}_1(s)$;
		$\mu(s,1)$ is the root of $z=\tilde{b}_1\left(s + (1-z)\lambda_1\right)$ with the smallest absolute value, cf. \cite[ p.~250]{cohen2012single}).
		A lengthy but straightforward calculation, that involves interchanging the integrations, yields, for $\mathrm{Re}~s \geq 0$,
		\begin{align}
		\int_{t=v}^{\infty}s c_1 e^{-c_1t} \int_{u=0}^{t - v}&e^{(s - \lambda_1(1 - \tilde{b}_1(s)))(t - u - v)}
		\mathbb{P}(V_1(u + v) = 0|V_1(0) = v)\mathrm{d}u \mathrm{d}t  \nonumber\\
		&= \frac{sc_1}{c_1 - s + \lambda_1(1- \tilde{b}_1(s))}\frac{e^{-(c_1 + (1- \mu(c_1, 1)) \lambda_1)v}}{c_1+(1- \mu(c_1, 1)) \lambda_1}.	
		\label{eq3.6}
		\end{align}
		Combining (\ref{eq3.4}) and (\ref{eq3.6}) yields
		\begin{equation}
		\begin{split}	
		\label{Step2}
		&\mathbb{E}(e^{-sV_1(T_1)}|V_1(0)=v) \\ &=
		\frac{c_1e^{-sv}}{c_1+\lambda_1(1-\tilde{b}_1(s))-s} -\frac{sc_1}{c_1-s+\lambda_1(1-\tilde{b}_1(s))}\frac{e^{-(c_1+(1-\mu(c_1,1))\lambda_1)v}}{c_1+(1-\mu(c_1,1))\lambda_1}.
		\end{split}
		\end{equation}
		
		\item{\bf Step 3:} Calculation of $\mathbb{E}(e^{-sV_1(T_1+T_2)}|V_1(0)=v)$.
		\begin{equation}
		\begin{split}		
		&\mathbb{E}(e^{-sV_1(T_1+T_2)}|V_1(0)=v)\\&= \int_{y=0}^{\infty}\mathbb{E}(e^{-sV_1(T_1+T_2)}|V_1(T_1)=y)f_{V_1}(V_1(T_1)=y|V_1(0)=v)\mathrm{d}y  \\
		&=\frac{c_2}{c_2+\lambda_1(1-\tilde{b}_1(s))}\int_{y=0}^{\infty} e^{-sy}f_{V_1}(V_1(T_1)=y|V_1(0)=v)\mathrm{d}y \\
		&=\frac{c_2}{c_2+\lambda_1(1-\tilde{b}_1(s))}\left[\frac{e^{-sv} c_1}{c_1+\lambda_1(1-\tilde{b}_1(s))-s}\right.\\
		&\left.
		-\frac{c_1 e^{-sv}}{c_1-s+\lambda_1(1-\tilde{b}_1(s))}\frac{e^{-(c_1+(1-\mu(c_1,1))\lambda_1)v}}{c_1+(1-\mu(c_1,1))\lambda_1}\right],
		\label{finalcondition}
		\end{split}
		\end{equation}
		where the second equation comes from Equation \eqref{Eq:firstterm} and the third from Equation \eqref{Step2}.

		\item{\bf Step 4:} Calculation of $\mathbb{E}(e^{-sV_1(T_1+T_2)})$ in steady-state.\\
		Observe that
		\begin{equation}
		\begin{split}	
		&\mathbb{E}(e^{-sV_1(T_1+T_2)}) \\&\int_{v=0}^{\infty}\mathbb{E}(e^{-sV_1(T_1+T_2)}|V_1(0)=v)f_{V_1(0)}(v) \mathrm{d}v\\
		&=\int_{v=0}^{\infty}	 \left[\frac{c_2}{c_2+\lambda_1(1-\tilde{b}_1(s))}\left[e^{-sv}\frac{c_1}{c_1+\lambda_1(1-\tilde{b}_1(s))-s}\right.\right.\\
		&-\left.\left.s\frac{c_1}{c_1-s+\lambda_1(1-\tilde{b}_1(s))}\frac{e^{-(c_1+(1-\mu(c_1,1))\lambda_1)v}}{c_1+(1-\mu(c_1,1))\lambda_1}\right]\right]f_{V_1(0)}(v) \mathrm{d}v,
		\label{18}
		\end{split}
		\end{equation}
		with $f_{V_1(0)}(v)$ the probability density function of $V_1(0)$.
		Now observe that in steady-state $V_1(T_1+T_2)$ has the same distribution as $V_1(0)$. So we can rewrite \eqref{18} as follows:
		\begin{align*}
		&\mathbb{E}(e^{-sV_1(T_1+T_2)})
		\\&=\int_{v=0}^{\infty}		 \left[\frac{c_2}{c_2+\lambda_1(1-\tilde{b}_1(s))}\left[e^{-sv}\frac{c_1}{c_1+\lambda_1(1-\tilde{b}_1(s))-s} \right.\right.\nonumber\\
		&\left.\left.-s\frac{c_1}{c_1-s+\lambda_1(1-\tilde{b}_1(s))}\frac{e^{-(c_1+(1-\mu(c_1,1))\lambda_1)v}}{c_1+(1-\mu(c_1,1))\lambda_1}\right]\right]f_{V_1(T_1+T_2)}(v) \mathrm{d}v.%\label{19.5}
		\end{align*}
		So,
		\begin{equation}
		\begin{split}	
		\mathbb{E}(		 &e^{-sV_1(T_1+T_2)})\left[\frac{c_2+\lambda_1(1-\tilde{b}_1(s))}{c_2}-\frac{c_1}{c_1+\lambda_1(1-\tilde{b}_1(s))-s}\right]\\
		=&-\frac{sc_1}{\left[c_1+\lambda_1(1-\tilde{b}_1(s))-s\right]\left(c_1+(1-\mu(c_1,1))\lambda_1\right)}
		\\&\mathbb{E}(e^{-\left(c_1+\left(1-\mu(c_1,1)\right)\lambda_1\right)V_1(T_1+T_2)})		 .
		\label{19}
		\end{split}
		\end{equation}
		Taking the limit as $s\rightarrow 0$ in \eqref{19} and using L'H\^{o}pital's rule yields
		\begin{align*}
		&\mathbb{E}(e^{-\left(c_1+\left(1-\mu(c_1,1)\right)\lambda_1\right)V_1(T_1+T_2)})
		\\&=-\frac{\left[\lambda_1 \mathbb{E}(B_1) c_1+\lambda_1 \mathbb{E}(B_1) c_2-c_2\right]\left(c_1+(1-\mu(c_1,1))\lambda_1\right)}{c_1c_2}.
		\end{align*}
		Hence,
		\begin{equation}
		\mathbb{E}(e^{-sV_1(T_1+T_2)})=\frac{s\left[\lambda_1 \mathbb{E}(B_1)(c_1+c_2)-c_2\right]}{\left[c_2+\lambda_1(1-\tilde{b}_1(s))\right]\left[c_1 + \lambda_1(1-\tilde{b}_1(s))-s\right]-c_1c_2}.
		\label{finaltotal}\end{equation}

		\item{\bf Step 5:} Calculation of $\mathbb{E}(e^{-sV_1})$ in steady-state.\\
		Firstly, let us denote by $S=1$ (respectively by $S=2$) the event of the server residing in the first (respectively second) queue. Then,
		\begin{align}		
		\mathbb{E}({e}^{-sV_1}) &= \mathbb{E}({e}^{-sV_1}| S=1) \mathbb{P}(S=1)+ \mathbb{E}({e}^{-sV_1}|S=2)\mathbb{P}(S=2) \nonumber\\
		&=\mathbb{E}({e}^{-sV_1}|S=1)\frac{c_2}{c_1+c_2}+ \mathbb{E}({e}^{-sV_1}|S=2)\frac{c_1}{c_1+c_2}.	 \label{Eq3.11}
		\end{align}
		Because of the memoryless property of the exponential distribution it is obvious that
		\[
		\mathbb{E}({e}^{-sV_1}| S=1) =\mathbb{E}(e^{-sV_1(T_1)}), \ \mathbb{E}({e}^{-sV_1}| S=2) = \mathbb{E}(e^{-sV_1(T_1+T_2)}).
		\]
		The latter term is given by \eqref{finaltotal}, while the former term is calculated using the same argument as in the derivation of Equation \eqref{Eq:firstterm}:
		\begin{equation}
		\mathbb{E}(e^{-sV_1(T_1+T_2)})= \mathbb{E}(e^{-sV_1(T_1)})\frac{c_2}{c_2 + \lambda_1 (1 - \tilde{b}_1(s))} .
		\label{eq3.15}
		\end{equation}
		Substituting  \eqref{eq3.15}, for $\mathbb{E}(e^{-sV_1(T_1)})$, and \eqref{finaltotal} in Equation \eqref{Eq3.11} yields \eqref{MLQ1}, which concludes the proof.
	\end{description}
	Similarly, we can also calculate the LST of the workload of the second queue.
\end{proof}
\begin{remark}
	It is not difficult to extend the above results to the case that the $T_2$ periods are non-exponential, see, e.g., \cite{gaver62,thiruvengadam63,avi-naor63},
	and to the case that the arrival process during those periods is a different compound Poisson process than during the $T_1$ periods, see, e.g., \cite{takinesengupta97} and \cite{fiems-maer-brun08}.
	One could even allow a more general non-decreasing L\'evy process (subordinator) during those $T_2$ periods.
	During $T_1$ periods, one could also allow the input process to be a subordinator.
	However, we do note that it is considerably more complicated to consider non-exponential $T_1$ periods, see, \cite{feder-green86}.
\end{remark}

%========== Section 4 : Decomposition, Heavy traffic limit and Asymptotics ==================%

\section{Workload decomposition and Heavy Traffic analysis}\label{sec:WorkloadDec}

In this section, we show that the steady-state workload $V_1$ (similarly for $V_2$) can be decomposed into two independent terms, one corresponding to the steady-state workload of the first queue in isolation, i.e., the M/G/1 queue with arrival rate $\lambda_1$ and service times $B_1$ (to be called: corresponding M/G/1 queue), and the second corresponding to the amount of work when the server is not serving the first queue, due to either an idle period or due to a visit at the second queue.\\
Assuming that the first three moments of $B_1$ are finite and then using the decomposition of $V_1$, we determine the mean, the variance, and the heavy traffic limit of the workload $V_1$. Furthermore, in this and in the next section, we use the decomposition to obtain various asymptotic (heavy traffic and/or heavy tail) results.
\begin{corollary}\label{Theo:DecompositionWorkload}
	The steady-state amount of work of the first queue, $V_1$, is distributed as a sum of two independent random variables $V_{\text{\tiny M/G/1}}$ and $Y$, i.e.,
	\begin{equation}\label{Eq:DecompositionWorkload}
	V_1 \overset{\text{d}}{=} V_{\text{\tiny M/G/1}} + Y,
	\end{equation}
	where $V_{\text{\tiny M/G/1}}$ is the steady-state amount of work in the corresponding M/G/1 queue, and $Y$ is the steady-state amount of work when the server is not serving at the first queue. The LST of the random variable $Y$ is given as
	\begin{equation}\label{Eq:D9}
	\begin{split}
	&\mathbb{E}(e^{-sY}) \\&=  \frac{c_2 - \rho_1 (c_1 + c_2)}{(1 - \rho_1)(c_1 + c_2)}\left[1 - \frac{s c_1}{(c_2 + c_1-s) \lambda_1(1 - \tilde{b}_1(s)) + \left(\lambda_1(1 - \tilde{b}_1(s))\right)^2  - s c_2}\right].
	\end{split}
	\end{equation}
\end{corollary}
\begin{proof}
	The workload decomposition result follows from \cite[Theorem~2.1]{boxma1989workloads}; it is readily verified that all conditions of that theorem are satisfied. And the LST of $Y$ can be directly obtained by dividing the LST of $V_1$ (which is given by Eqn. \eqref{MLQ1}) by the known LST of the M/G/1 queue (cf. \cite[p.~257, Equation~(4.90)]{cohen2012single}).
\end{proof}

\begin{remark}
$\mathbb{E}(e^{-sY})$ could also have been obtained by writing it as a weighted sum of two known terms:
(i) 1, which is the LST of the (zero) workload in the first queue during an idle part of the visit period, and
(ii) $\mathbb{E}(e^{-sV_1(T_1+T_2)})$, which is given in (\ref{finaltotal}). PASTA implies that the latter term is also the LST of the workload in the first queue
at an arbitrary epoch of a visit period of the other queue.
\end{remark}

We now use the decomposition result \eqref{Eq:DecompositionWorkload} to determine the mean and the variance of $V_1$.
\begin{theorem}\label{Theo:Mean_Var}
	The expectation of the steady-state workload of the first queue, $\mathbb{E}(V_1)$, is
	\begin{equation}\label{Eq:Mean}
	\mathbb{E}(V_1) = \frac{\rho_1(c_1 + c_2)}{c_2 - \rho_1(c_1 + c_2)} \left[\frac{ 1}{2}  \frac{\mathbb{E}(B_1^2)}{\mathbb{E}(B_1)} + \frac{c_1 }{(c_1 + c_2)^2}\right],
	\end{equation}
	and the corresponding variance, $\mathbb{V}\mathrm{ar}(V_1)$, is
	\begin{equation}
	\begin{split}
	\mathbb{V}\mathrm{ar}(V_1) =& \frac{\rho_1(c_1 + c_2)}{c_2 - \rho_1(c_1 + c_2)} \\ &\left[\frac{1}{3} \frac{\mathbb{E}(B_1^3)}{\mathbb{E}(B_1)} + \frac{1}{4} \frac{\rho_1 }{1 - \rho_1} \frac{\left(\mathbb{E}(B_1^2)\right)^2}{\left(\mathbb{E}(B_1)\right)^2} + \frac{c_1}{(c_1 + c_2)^2} \frac{\mathbb{E}(B_1^2)}{\mathbb{E}(B_1)} + \frac{c_1}{(c_1 + c_2)^3}\right]. \label{Eq:Var}
	\end{split}
	\end{equation}	
\end{theorem}

\begin{proof}
	The mean and variance can be obtained by using the decomposition result \eqref{Eq:DecompositionWorkload}. For this purpose, we can separately calculate the mean and the variance of the M/G/1 queue, cf. \cite[p.~256]{cohen2012single}, as well as the mean and the variance corresponding to the random variable $Y$. For the latter we use Equation \eqref{Eq:D9}
	(after dividing its numerator and denominator by $s$). Combining these results yields Equations \eqref{Eq:Mean} and \eqref{Eq:Var}.	
\end{proof}

\begin{remark}
	Equation \eqref{Eq:Mean} and Equation \eqref{Eq:Var} for $c_2 \to \infty$ (or equivalently for $c_1 \to 0$) yield the corresponding expressions for the mean and the variance of the M/G/1 queue, cf. \cite[p.~256]{cohen2012single}.
\end{remark}

Now, we study the behavior of the workload $V_1$ in heavy traffic, i.e., when $\rho_1 \uparrow \frac{c_2}{c_1 + c_2}$. In Corollary \ref{Theo:DecompositionWorkload} we have shown that $V_1$ can be written as the sum of the independent random variables $V_{\text{\tiny M/G/1}}$ and $Y$. Since most of the results related to the M/G/1 queue are already known, we take a closer look at $\mathbb{E}(e^{-sY})$, with the assumption that the first three moments of $B_1$ are finite. Substituting $\tilde{b}_1(s) = 1 - s \mathbb{E}(B_1) + \frac{s^2}{2} \mathbb{E}(B_1^2) - \frac{s^3}{3!} \mathbb{E}(B_1^3) + o(s^3)$ in \eqref{Eq:D9} and rearranging the terms yields, for $s \downarrow 0$,
\begin{align}\label{Eq:HTraffic5}
\mathbb{E}(e^{-sY})=  \frac{A_0}{1 - \rho_1} \left[1 - \frac{c_1}{c_1 + c_2} \left(\frac{1}{ A_0 + s A_1 - \frac{s^2}{2} A_2 + o(s^2)}\right) \right],
\end{align}
with
\begin{align}
A_0 =& \frac{c_2 }{c_1 + c_2} - \rho_1,\ A_1 = \frac{\rho_1}{c_1 + c_2}\left(1 - \rho_1 + \frac{c_1 + c_2}{2} \frac{\mathbb{E}(B_1^2)}{\mathbb{E}(B_1)}\right),\label{A0A1}\\
A_2 =& \frac{\rho_1^2}{c_1 + c_2}\left(\frac{1 - 2\rho_1}{\rho_1} \frac{\mathbb{E}(B_1^2)}{\mathbb{E}(B_1)} + \frac{c_1 + c_2}{3 \rho_1} \frac{\mathbb{E}(B_1^3)}{\mathbb{E}(B_1)}\right).
\label{A2}
\end{align}
Equation \eqref{Eq:HTraffic5} will play a very important role in the proof of Theorem \ref{Theo:Heavy_traffic} and also in the next section where we study the tail behavior of the workload $V_1$.
\begin{theorem}\label{Theo:Heavy_traffic}
	Assume that $\mathbb{E}(B_1^2) < \infty$.
	For $\rho_1 \uparrow \frac{c_2}{c_1 + c_2}$,
	\begin{equation}\label{Eq:HTraffic}
	\left( \frac{c_2}{c_1 + c_2} - \rho_1 \right) V_1 \overset{d}{\longrightarrow} Z, \end{equation} 	
	with $Z$ an exponentially distributed random variable with mean $\frac{c_1 c_2}{(c_1 + c_2)^3}  + \frac{c_2}{c_1 + c_2} \frac{1}{2}\frac{\mathbb{E}(B_1^2)}{\mathbb{E}(B_1)}$.
\end{theorem}

\begin{proof} To obtain the heavy traffic limit of $V_1$ one can use the workload decomposition. Corollary \ref{Theo:DecompositionWorkload} implies that
	\begin{equation}\label{Eq:HTraffic0}
	\mathbb{E}(e^{-sV_1}) = \mathbb{E}(e^{-sV_{\text{\tiny M/G/1}}}) \mathbb{E}(e^{-sY}).
	\end{equation}
	Replacing $s$ by $s A_0 = s(\frac{c_2}{c_1+c_2}-\rho_1)$, cf. \eqref{A0A1}, in the above equation and taking the limit $\rho_1 \uparrow \frac{c_2}{c_1 + c_2}$ yields
	\begin{equation}\label{Eq:HTraffic1}
	\begin{split}
	&\lim\limits_{\rho_1 \uparrow \frac{c_2}{c_1 + c_2}} \mathbb{E}\left(e^{-s \left(\frac{c_2}{c_1 + c_2} - \rho_1 \right)V_1}\right) \\ &= \lim\limits_{\rho_1 \uparrow \frac{c_2}{c_1 + c_2}} \mathbb{E}\left(e^{-s\left(\frac{c_2}{c_1 + c_2} - \rho_1 \right)V_{\text{\tiny M/G/1}}}\right) \mathbb{E}\left(e^{-s \left(\frac{c_2}{c_1 + c_2} - \rho_1 \right)Y}\right).
	\end{split}
	\end{equation}
	The first term in the right hand side obviously tends to one for $\rho_1 \uparrow \frac{c_2}{c_1+c_2}$,
	as the corresponding M/G/1 queue is in heavy traffic only when $\rho_1 \uparrow 1$.
	In order to calculate the limit for the second term in \eqref{Eq:HTraffic1}, we replace $s$ by $s A_0 = s(\frac{c_2}{c_1+c_2}-\rho_1)$, cf. \eqref{A0A1}, in \eqref{Eq:HTraffic5}, which yields
	\begin{equation}\label{Eq:HTraffic7}
	\mathbb{E}(e^{-s(\frac{c_2}{c_1+c_2}-\rho_1) Y}) =  \frac{1}{1 - \rho_1}\left[A_0 + \frac{c_1}{c_1 + c_2} \left( \frac{1}{1 + s A_1 - \frac{s^2}{2} A_0 A_2 + o\left(s^2A_0\right)}\right)\right].
	\end{equation}
	Taking the limit $\rho_1 \uparrow \frac{c_2}{c_1 + c_2}$ in \eqref{Eq:HTraffic7} yields
	\begin{equation}\label{Eq:HTraffic8}	
	\lim\limits_{\rho_1 \uparrow \frac{c_2}{c_1 + c_2}} \mathbb{E}(e^{-s \left(\frac{c_2 }{c_1 + c_2} - \rho_1\right) Y}) = \frac{1}{1 + s A_1},	
	\end{equation}
	with $A_1 $ given in \eqref{A0A1}. From \eqref{Eq:HTraffic8}, \eqref{A0A1}, and \eqref{Eq:HTraffic1} the statement of the theorem follows by noticing that the right hand side of \eqref{Eq:HTraffic8} corresponds to the LST of an exponentially distributed random variable with mean $A_1$.
\end{proof}

\begin{remark}
	Letting $c_2 \to \infty$, Theorem \ref{Theo:Heavy_traffic} indicates that the heavy traffic result reduces to that of an ordinary M/G/1 queue.
\end{remark}

%========================================= Heavy traffic limit =============================================%

\section{Heavy tail asymptotics}\label{sec:Asympt}
In this section, we discuss the tail behavior of the workload in the case of heavy tailed service time distributions. We also study the heavy traffic behavior of the workload $V_1$ when the service time distribution $B_1$ is regularly varying. To do this analysis, we now introduce the definition of a regularly varying random variable/distribution.
\begin{definition}
	The distribution function of a random variable $B_1$ on $[0, \infty)$ is called regularly varying of index $-\nu$, with $\nu \in \mathbb{R}$, if
	\begin{equation}\label{RVDef}	
	\mathbb{P}(B_1 > x) \sim L(x) x^{-\nu},~ {x \uparrow \infty},
	\end{equation}
	with $L(x)$ a slowly varying function at infinity, i.e., $\lim\limits_{x \to \infty} \frac{L(\alpha x)}{L(x)} = 1$, for all $\alpha > 1.$	
\end{definition}

\begin{theorem}\label{Theo:Heavy_tail}
	If the service time distribution of the random variable $B_1$ is regularly  varying of index $-\nu$, with $\nu \in (1, 2)$, then the workload of the first queue under the stability condition \eqref{Stability} is regularly varying at infinity of index $1 - \nu$. More precisely,
	\begin{equation}\label{Eq:Heavy_tail}
	\mathbb{P}\left(V_1 > x\right) \sim \frac{\rho_1}{\frac{c_2}{c_1+c_2} - \rho_1}  \frac{1}{\mathbb{E}(B_1)(\nu-1)} x^{1 - \nu} L\left(x\right), \quad \quad x \uparrow \infty.
	\end{equation}
\end{theorem}

\begin{proof}
	To prove that $V_1$ is regularly  varying at infinity, one can again use the decomposition property of the workload $V_1$. From Corollary \ref{Theo:DecompositionWorkload}, we get
	\begin{equation}\label{Eq:Heavy_tail1}
	\mathbb{P}(V_1 > x) = \mathbb{P}(V_{\text{\tiny M/G/1}} + Y > x).
	\end{equation}
	In the M/G/1 queue it follows from \cite{10.2307/3212351} that  $\mathbb{P}(V_{\text{\tiny M/G/1}} > x)$ is regularly varying of index $1 - \nu$ at infinity if and only if the tail of the service time distribution $\mathbb{P}(B_1 > x)$ is regularly varying of index $-\nu$ at infinity, and one has
	\begin{equation}\label{Eq:Heavy_tail2}
	\mathbb{P}\left(V_{\text{\tiny M/G/1}} > x\right) 	\sim	\frac{\rho_1}{\rho_1 - 1} \frac{1}{\mathbb{E}(B_1)(1 - \nu)}x^{1 - \nu}
	L\left(x\right), \quad \quad \quad x \uparrow \infty.
	\end{equation}
	Now we have to compute $\mathbb{P}(Y > x)$ for $x \uparrow \infty$.
	Our main tool is the Tauberian theorem of \cite[Theorem~8.1.6]{bingham1989regular},
	which relates the behavior of a regularly varying function at infinity and the behavior of its LST near $0$.
	Applying this theorem to \eqref{RVDef} gives
	\begin{equation*}%\label{Eq:Heavy_tail3}
	\tilde{b}_1(s) - 1 + s \mathbb{E}(B_1) \sim - \Gamma\left(1 - \nu\right) s^{\nu} L\left(\frac{1}{s}\right), \quad \quad s \downarrow 0,
	\end{equation*}
	and hence
	\begin{equation}\label{Eq:Heavy_tail31}
	\frac{\lambda_1 \left(1 - \tilde{b}_1(s)\right)}{s} = \rho_1 \left(1 + \frac{\Gamma\left(1 - \nu\right)}{\mathbb{E}(B_1)} s^{\nu - 1} L\left(\frac{1}{s}\right)\right), \quad \quad s \downarrow 0.
	\end{equation}
	Substituting Equation \eqref{Eq:Heavy_tail31} in \eqref{Eq:D9} yields, for $s \downarrow 0$:
	\begin{align*}  	
	&\mathbb{E}(e^{-sY}) \\ &=  \frac{c_2 - \rho_1 (c_1 + c_2)}{(1 - \rho_1)(c_1 + c_2)}\left[1 - \frac{c_1}{(c_1 + c_2) \rho_1 \left(1 + \frac{\Gamma\left(1 - \nu\right)}{\mathbb{E}(B_1)} s^{\nu - 1} L\left(\frac{1}{s}\right)\right) - c_2 + O(s)}\right]\nonumber\\
	&=  \frac{c_2 - \rho_1 (c_1 + c_2)}{(1 - \rho_1)(c_1 + c_2)}\nonumber\\
	&\times \left[1 + \frac{c_1}{\left(c_2 - \rho_1 (c_1 + c_2)\right) \left( 1 -  \frac{\rho_1 (c_1 + c_2)}{c_2 - \rho_1 (c_1 + c_2)} \frac{\Gamma\left(1 - \nu\right)}{\mathbb{E}(B_1)}  s^{\nu - 1} L\left(\frac{1}{s}\right)\right) + O(s)}\right] \nonumber\\	
	&=  \frac{c_2 - \rho_1 (c_1 + c_2)}{(1 - \rho_1)(c_1 + c_2)}\nonumber\\
	&+ \frac{c_1}{(1 - \rho_1)(c_1 + c_2)}\left( 1 +  \frac{\rho_1 (c_1 + c_2)}{c_2 - \rho_1 (c_1 + c_2)} \frac{\Gamma\left(1 - \nu\right)}{\mathbb{E}(B_1)}  s^{\nu - 1} L\left(\frac{1}{s}\right) + O(s) \right).%\label{Eq:Heavy_tail5}	
	\end{align*}
	Simplifying, we get
	\begin{equation*}%\label{Eq:Heavy_tail6}	
	\mathbb{E}(e^{-sY}) -  1 =  \frac{\rho_1 c_1}{(1 - \rho_1)\left(c_2 - \rho_1 (c_1 + c_2) \right)} \frac{\Gamma\left(1 - \nu\right)}{\mathbb{E}(B_1)} s^{\nu - 1} L\left(\frac{1}{s}\right),  \quad \quad  s \downarrow 0.	
	\end{equation*}  	
	Applying the Tauberian theorem of \cite[Theorem~8.1.6]{bingham1989regular} once again, now in the reverse direction, yields
	\begin{align}
	\mathbb{P}\left(Y > x\right) \sim& - \frac{1}{\Gamma\left(2 - \nu\right)} \frac{\rho_1 c_1}{(1 - \rho_1)\left(c_2 - \rho_1 (c_1 + c_2) \right)} \frac{\Gamma\left(1 - \nu\right)}{\mathbb{E}(B_1)} x^{1 - \nu} L\left(x\right) \nonumber \\
	=& \frac{\rho_1 c_1 }{(1 - \rho_1)\left(c_2 - \rho_1 (c_1 + c_2)\right)} \frac{1}{\mathbb{E}(B_1)(\nu - 1)} x^{1 - \nu} L\left(x\right), \quad \quad  \quad x \uparrow \infty.\label{Eq:Heavy_tail7}
	\end{align}
	From Equation \eqref{Eq:Heavy_tail2} and  \eqref{Eq:Heavy_tail7}, we see that both $V_{\text{\tiny M/G/1}}$ and $Y$ are regularly varying random variables of index $1 - \nu$. Using the workload decomposition property \eqref{Eq:DecompositionWorkload} and a well known result regarding the tail behavior of the sum of two independent regularly  varying random variables of the same index, see \cite{sigman1999appendix}, yields
	\begin{eqnarray}\label{Eq:Heavy_tail8}
	\mathbb{P}\left(V_1 > x\right) \sim (C_1 + C_2) x^{1 - \nu} L\left(x\right), \quad\quad\quad \quad \quad \quad x \uparrow  \infty,
	\end{eqnarray}	  	
	with $C_1$ and $C_2$ the coefficients of the tail $x^{1 - \nu}$ for  $V_{\text{\tiny M/G/1}}$ and $Y$ in \eqref{Eq:Heavy_tail2} and \eqref{Eq:Heavy_tail7}, respectively. Substituting the coefficients from \eqref{Eq:Heavy_tail2} and  \eqref{Eq:Heavy_tail7} concludes the proof of the theorem.
\end{proof}

\begin{remark}
	Letting $c_2 \to \infty$ in Equation \eqref{Eq:Heavy_tail8} yields
	\begin{eqnarray}\label{Eq:Heavy_tail_remark}
	\mathbb{P}\left(V_1 > x\right) 	\sim \frac{\rho_1}{1 - \rho_1} \frac{1}{\mathbb{E}(B_1)(\nu - 1)}x^{1 - \nu} L\left(x\right),	 \quad \quad x \uparrow \infty,
	\end{eqnarray}
	which is, as expected, the result for an ordinary M/G/1 queue.
\end{remark}

\begin{remark}
	Theorem~\ref{Theo:Heavy_tail} is closely related to \cite[Theorem~4.1]{BoxmaKurkova} for a single server queue
	with alternating high and low service speeds. In \cite{BoxmaKurkova} both the service time distribution {\em and} the distribution of the periods of low service speed are regularly varying. If the latter tail is less heavy than the tail of the service time distribution, then our formula (\ref{Eq:Heavy_tail}) displays exactly the same tail behavior as \cite[Formula~(4.1)]{BoxmaKurkova}.
\end{remark}

\noindent
In the next theorem  we discuss how Theorem~\ref{Theo:Heavy_tail} can be generalized to the case of subexponential (residual) service times.

\begin{definition}
	A distribution function $\mathbb{P}(B_1 \leq x)$, $x \geq 0$, is called subexponential if
	\[ \mathbb{P}(B_{11} + \cdots + B_{1n} > x) \sim n \mathbb{P}(B_{11} > x),\ x \uparrow \infty ,\]
	for any $n \geq 2$, with $B_{11}, \dots, B_{1n}$ independent and identical copies of $B_1$.
\end{definition}

%
%\begin{theorem}\label{Theo:SubExp}
%	If the distribution of the residual service time requirement, say $B_1^{r}$, is subexponential, then $V_1$ is also subexponential and
%	\begin{align} \label{Eq:SubExp}
%	\mathbb{P}\left(V_1 > x\right) \sim  \frac{\rho_1}{\frac{c_2}{c_1+c_2} - \rho_1 } \mathbb{P}\left(B_1^{r} > x\right),\ x \uparrow \infty.\end{align}
%\end{theorem}

%\noindent
%It can be shown that a similar result as in Theorem~\ref{Theo:Heavy_tail} holds for subexponential distributions.
%
\begin{theorem}\label{Theo:SubExp}
If the distribution of the residual service time requirement, say $B_1^{r}$, is subexponential, then $V_1$ is also subexponential and
\begin{align} \label{Eq:SubExp}\mathbb{P}\left(V_1 > x\right) \sim  \frac{\rho_1}{\frac{c_2}{c_1+c_2} - \rho_1 } \mathbb{P}\left(B_1^{r} > x\right),\ x \uparrow \infty.\end{align}
\end{theorem}

%\begin{proof}[Heuristic proof]
\noindent
{\em Heuristic proof.}
The asymptotic relation in~\eqref{Eq:SubExp} can be proved formally using sample-path techniques along the following lines.
We assume the system is in stationarity and focus on the workload at time $t = 0$. If the workload level  $V_1$ at this time is very large, then that is most likely due to the prior arrival of a customer with a large service requirement $B_1$, at  some time $t = -y$.  We can observe that from time $t = -y$ onward, the workload decreases nearly linearly with rate  $\frac{c_2}{c_1+c_2} - \rho_1$. So in order for the workload at time $t = 0$ to exceed the level $x$, the service requirement $B_1$ must be larger than $x + y\left(\frac{c_2}{c_1+c_2} - \rho_1\right)$. Since customers arrive according to a Poisson process with rate $\lambda_1$, the distribution of the workload $V_1$ for large $x$ can be computed as
\begin{eqnarray}\label{Eq:Alt_SubExp1}
\mathbb{P}\left(V_1 > x\right) \sim \int_{y=0}^{\infty} \mathbb{P}\left(B_1 > x + y\left(\frac{c_2}{c_1+c_2} - \rho_1 \right)\right) \lambda_1  \mathrm{d}y.
\end{eqnarray}
A change of variable $z :=  x + y\left(\frac{c_2}{c_1+c_2} - \rho_1 \right)$ in Equation \eqref{Eq:Alt_SubExp1} yields
\begin{align}
\mathbb{P}\left(V_1 > x\right) \sim& \frac{\lambda_1}{\frac{c_2}{c_1+c_2} - \rho_1} \int_{z = x}^{\infty} \mathbb{P}\left(B_1 > z \right)   \mathrm{d}z \nonumber\\
=& \frac{\lambda_1 \mathbb{E}(B_1)}{\frac{c_2}{c_1+c_2} - \rho_1} \int_{z = x}^{\infty} \frac{\mathbb{P}\left(B_1 > z \right) }{\mathbb{E}(B_1)}  \mathrm{d}z
= \frac{\rho_1}{\frac{c_2}{c_1+c_2} - \rho_1}\mathbb{P}\left(B^{r}_1 > x \right),
\label{Eq:Alt_SubExp2}
\end{align}
which leads to Relation~\eqref{Eq:SubExp}.\\
%\end{proof}
This proof can be made rigorous by providing lower and upper bounds for $\mathbb{P}(V_1 > x)$
that in the limit coincide. The lower bound  is easily obtained by using the law of large numbers. The upper bound is more difficult; one has to give a formal version of the statement {\em ``exceedance of a high level $x$ happens as a consequence of a single big jump''}, and one has to show that other exceedance scenarios (like two rather big jumps) do not contribute to the asymptotics of the exceedance probability. We refer to \cite[Section~2.4]{zwart2001queueing} for a detailed exposition of this technique.

\begin{remark}
	Note that, indeed, Relation~\eqref{Eq:SubExp} contains the result of Theorem \ref{Theo:Heavy_tail} as a special case, since $B_1$ being regularly varying at infinity of index $-\nu$, with $\nu \in (1, 2)$, has a subexponential distribution. In this regularly varying case, we have
	\begin{equation*}
	\mathbb{P}\left(B^{r}_1 > x \right) =  \int_{z = x}^{\infty} \frac{\mathbb{P}\left(B_1 > z \right) }{\mathbb{E}(B_1)}  \mathrm{d}z \sim  \frac{1}{{\mathbb{E}(B_1)}}\int_{z = x}^{\infty} z^{-\nu} L(z)  \mathrm{d}z,\ x \uparrow \infty.
	\end{equation*}
	In the above equation by applying the regular varying function property from \cite[p.~26,~Proposition 1.5.8]{bingham1989regular}, we get
	\begin{equation}\label{Eq:Alt_SubExp3}
	\mathbb{P}\left(B^{r}_1 > x \right)  \sim \frac{1}{\mathbb{E}(B_1)(\nu-1)} x^{1 - \nu} L(x),\ x \uparrow \infty.
	\end{equation}
	Combining \eqref{Eq:Alt_SubExp2} and \eqref{Eq:Alt_SubExp3}, we obtain Theorem \ref{Theo:Heavy_tail}.
\end{remark}

\noindent
Now we concentrate on a heavy traffic limit theorem for $V_1$ in the heavy tailed case. To do this analysis, we first scale $V_1$ by the coefficient of contraction $\Delta(\rho_1)$. Similarly to \cite[p.~188,~Equation~(4.24)]{boxma1999heavy}, we define the coefficient of contraction $\Delta(\rho_1)$ as the unique root of the following equation in $x$
\begin{equation}\label{Eq:coefficients_contraction}
x^{\nu - 1} L\left(\frac{1}{x}\right) = \frac{\frac{c_2}{c_1 + c_2} - \rho_1}{\rho_1}, ~ x > 0,
\end{equation}
such that $\Delta(\rho_1) \downarrow 0$ for $\rho_1 \uparrow \frac{c_2}{c_1 + c_2}$.

\begin{theorem}\label{Theo:HeavyTraffic_Htail}
	If the service time distribution of the random variable $B_1$ is regularly varying of index $-\nu$, with $\nu \in (1, 2)$, then the heavy traffic limiting distribution of workload $V_1$ of the first queue in the heavy tailed case is given by the Mittag-Leffler distribution:
	\begin{equation}\label{Eq:HeavyTraffic_Htail}
	\lim\limits_{\rho_1 \uparrow \frac{c_2}{c_1 + c_2}} \mathbb{E}\left(e^{-s \Delta(\rho_1)V_1}\right) = \frac{1}{1 + (\mathbb{E}(B_1))^{\nu - 1}s^{\nu - 1}}.
	\end{equation} 	  	
\end{theorem}

\begin{proof} We can obtain  the heavy traffic limit in the heavy tailed case by using the workload decomposition property \eqref{Theo:DecompositionWorkload} and its LST version (\ref{Eq:HTraffic0}). The heavy traffic limit can be computed by replacing $s$ by $s \Delta(\rho_1)$ in Equation \eqref{Eq:HTraffic0} and taking the limit $\rho_1 \uparrow \frac{c_2}{c_1 + c_2}$, which yields
	\begin{equation}\label{Eq:HeavyTraffic_Htail1}
	\lim\limits_{\rho_1 \uparrow \frac{c_2}{c_1 + c_2}} \mathbb{E}\left(e^{-s \Delta(\rho_1)V_1}\right) = \lim\limits_{\rho_1 \uparrow \frac{c_2}{c_1 + c_2}} \mathbb{E}\left(e^{-s \Delta(\rho_1) V_{\text{\tiny M/G/1}}}\right) \mathbb{E}\left(e^{-s \Delta(\rho_1)Y}\right).
	\end{equation}
	Just as in the light tailed case (cf. Theorem \ref{Theo:Heavy_traffic}),
	the contribution of $V_{\text{\tiny M/G/1}}$ becomes negligible compared to the contribution of $Y$. To calculate the limit for the second factor in \eqref{Eq:HeavyTraffic_Htail1}, we use \eqref{Eq:D9}.
	\begin{equation}\label{Eq:HeavyTraffic_Htail3}
	\begin{split}
	&\mathbb{E}\left(e^{-s\Delta(\rho_1)Y} \right) \\ &= \frac{\frac{c_2 }{c_1 + c_2} - \rho_1}{1 - \rho_1} \left[1 - \frac{c_1}{c_1 + c_2} \frac{1}{\frac{f(s\Delta(\rho_1))}{s\Delta(\rho_1)} - \frac{f(s\Delta(\rho_1))}{c_1 + c_2} + \frac{s\Delta(\rho_1)}{c_1 + c_2}\left(\frac{f(s\Delta(\rho_1))}{s\Delta(\rho_1)}\right)^2 - \frac{c_2}{c_1 + c_2}}\right],
	\end{split}
	\end{equation}
	with $f(s\Delta(\rho_1)) = \frac{\rho_1(1 - \tilde{b}_1(s\Delta(\rho_1)))}{\mathbb{E}(B_1)}$. Taking the limit $\rho_1 \uparrow \frac{c_2}{c_1 + c_2}$ in \eqref{Eq:HeavyTraffic_Htail3} yields
	\begin{equation}\label{Eq:HeavyTraffic_Htail4}
	\lim\limits_{\rho_1 \uparrow \frac{c_2}{c_1 + c_2}} \mathbb{E}\left(e^{-s\Delta(\rho_1)Y} \right) = 	 -\lim\limits_{\rho_1 \uparrow \frac{c_2}{c_1 + c_2}} \frac{c_1\left(\frac{c_2 }{c_1 + c_2} - \rho_1\right)}{(c_1 + c_2)(1 - \rho_1)} \frac{1}{\frac{f(s\Delta(\rho_1))}{s\Delta(\rho_1)} - \frac{c_2}{c_1 + c_2}},
	\end{equation}	
	since $ \frac{s\Delta(\rho_1)}{c_1 + c_2}\left(\frac{f(s\Delta(\rho_1))}{s\Delta(\rho_1)}\right)^2 \to 0$, $f(s\Delta(\rho_1)) \to 0$ and $\Delta(\rho_1) \to 0$ when $\rho_1 \uparrow \frac{c_2}{c_1 + c_2}$. After rearranging the terms of \eqref{Eq:HeavyTraffic_Htail4} we get
	\begin{equation}\label{Eq:HeavyTraffic_Htail5}
	\begin{split}
	&\lim\limits_{\rho_1 \uparrow \frac{c_2}{c_1 + c_2}} \mathbb{E}\left(e^{-s\Delta(\rho_1)Y} \right) \\ &= -\lim\limits_{\rho_1 \uparrow \frac{c_2}{c_1 + c_2}} \frac{c_1}{(c_1 + c_2)(1 - \rho_1)} \frac{1}{\frac{1}{\frac{c_2}{c_1 + c_2} - \rho_1} \left[\frac{f(s\Delta(\rho_1))}{s\Delta(\rho_1)} - \frac{c_2}{c_1 + c_2} \right]}.
	\end{split}
	\end{equation}
	Since $B_1$ is regularly varying, we get by using \cite[Lemma~5.1~(iv)]{boxma1999heavy},
	\begin{equation}\label{Eq:HeavyTraffic_Htail6}
	\begin{split}
	&\lim\limits_{\rho_1 \uparrow \frac{c_2}{c_1 + c_2}} \frac{1}{\frac{c_2}{c_1 + c_2} - \rho_1} \left[\frac{f(s\Delta(\rho_1))}{s\Delta(\rho_1)} - \frac{c_2}{c_1 + c_2} \right] \\ &= - \lim\limits_{\rho_1 \uparrow \frac{c_2}{c_1 + c_2}}  \left(1 + \frac{\rho_1}{\frac{c_2}{c_1 + c_2} - \rho_1} \left[1 - \frac{1 - \tilde{b}_1(s\Delta(\rho_1))}{s\mathbb{E}(B_1)\Delta(\rho_1)}\right]\right).	
	\end{split}
	\end{equation}
	Using \cite[p.~188,~Equation~(4.22)]{boxma1999heavy}, we know that
	\begin{equation}\label{Eq:HeavyTraffic_Htail7}
	1 - \frac{1 - \tilde{b}_1(s\Delta(\rho_1))}{s\mathbb{E}(B_1)\Delta(\rho_1)} \sim (\mathbb{E}(B_1) s \Delta(\rho_1))^{\nu - 1} L\left(\frac{1}{s\mathbb{E}(B_1)\Delta(\rho_1)}\right),\ s \downarrow 0.
	\end{equation}
	From the definition of the coefficient of contraction $\Delta(\rho_1)$ as the unique root of Equation \eqref{Eq:coefficients_contraction} such that $\Delta(\rho_1) \downarrow 0$ for $\rho_1 \uparrow \frac{c_2}{c_1 + c_2}$, we have
	\begin{equation}\label{Eq:HeavyTraffic_Htail8}
	(\Delta(\rho_1))^{\nu - 1} L\left(\frac{1}{\Delta(\rho_1)}\right) = \frac{\frac{c_2}{c_1 + c_2} - \rho_1}{\rho_1}.
	\end{equation}
	Furthermore, from the definition of a slowly varying function $L(\cdot)$, we get $\frac{L\left(\frac{1}{s\mathbb{E}(B_1)\Delta(\rho_1)}\right) }{L\left(\frac{1}{\Delta(\rho_1)}\right)} \to 1$, when $\Delta(\rho_1) \downarrow 0$.
	Now by combining \eqref{Eq:HeavyTraffic_Htail5} - \eqref{Eq:HeavyTraffic_Htail8} we get
	\begin{equation}\label{Eq:HeavyTraffic_Htail10}
	\lim\limits_{\rho_1 \uparrow \frac{c_2}{c_1 + c_2}} \mathbb{E}\left(e^{-s\Delta(\rho_1)Y} \right)	= \frac{1}{1 + (\mathbb{E}(B_1))^{\nu - 1} s^{\nu - 1}}.
	\end{equation}
	Substituting \eqref{Eq:HeavyTraffic_Htail10} in Equation \eqref{Eq:HeavyTraffic_Htail1} concludes the proof of the theorem.
\end{proof}

\begin{remark}
	In \cite{boxma1999heavy} a class of service time distributions is considered that is slightly larger than the class of regularly varying distributions. Theorem  \ref{Theo:HeavyTraffic_Htail} can be seen to hold under these conditions as well.
\end{remark}

\begin{remark}
At this stage it may be appropriate to discuss some of the advantages and disadvantages of working with timers instead of
traditional polling disciplines like exhaustive, gated or $k$-limited.
An {\em advantage} of working with timers appears to be the following.
It is proven in \cite{BDR} for $N$-queue cyclic polling models with exhaustive or gated service,
that if the heaviest service time distribution (say, at $Q_M$) is regularly varying of index $- \nu$, then {\em all} waiting time distributions are regularly varying of index $1-\nu$. This is intuitively clear: There is a positive probability of a customer in $Q_1$ arriving during a service time of $Q_M$, and then its waiting time includes a residual type-$M$ service time -- which
is regularly varying of index $1-\nu$.
That intuition also indicates that a similar tail behavior will occur for disciplines like $k$-limited.

On the contrary, in our model with exponential timers, the service time distribution at one queue has no effect at all on the waiting time distribution or workload distribution at the other queue (for the workload, this is also seen from Formula (\ref{MLQ1}), which does not involve $\tilde{b}_2(s)$).
In particular, the waiting time distribution at $Q_1$ will asymptotically behave exponentially if the service time tail is exponential; and if the service time distribution at $Q_1$ is regularly varying of index $-\zeta$ while the service time distribution at $Q_2$ is regularly varying of index $-\nu$
with $\zeta > \nu$, then the waiting time distribution will be regularly varying of index $1-\zeta < 1 - \nu$.

Next to protection against heavy tails at other queues, there is also some protection
against long mean service times at other queues.
If $\E B_2$ is much larger than $\E B_1$, then a type-$1$ customer in an ordinary polling model may experience a long mean delay because of the presence of type-$2$ customers. However, when timers are used, a customer of type $1$ with a short service time will not suffer much from the presence of type-$2$ customers with long mean service times. This is a similar protection phenomenon as round robin or processor sharing protecting short customers against having to wait long for a customer with a very long service time.

A {\em disadvantage} of using timers is that the system is not work conserving, and hence operates in a sense less efficient than one would wish. This is revealed in the fully symmetric case.
It follows from (\ref{Eq:Mean}) that the mean workload at $Q_1$, in the fully symmetric case (equal arrival rates, service time
distributions
and mean visit periods), equals
\begin{equation}
\E V_1 = \frac{\rho}{1-\rho} \frac{\E B_1^2}{2 \E B_1} + \frac{1}{4c_1} ,
\end{equation}
so that the mean total workload equals
\begin{equation}
\E V_1 + \E V_2 = 2 \frac{\rho}{1-\rho} \frac{\E B_1^2}{2 \E B_1} + \frac{1}{2c_1} .
\end{equation}
On the other hand, in a symmetric two-queue polling model with, e.g., exhaustive or gated service,
one has work conservation, and hence the mean total workload equals
\begin{equation}
\E V_{total} = \frac{\rho}{1-\rho} \frac{\E B_1^2}{2 \E B_1}.
\end{equation}
This is less than half the value of $\E V_1 + \E V_2$. The main reason for this is the fact that the server is sometimes idle although there is work at the other queue. This even holds when $c_1 =c_2 \rightarrow \infty$;
%WE SHOULD ADD A SENTENCE WHY EVEN IN THE LIMIT THERE IS THIS FACTOR $2$ FOR THE MEAN WORKLOAD.
$Q_1$ operates effectively as if the server works at half speed for it, or as if the service times are twice as long.
\end{remark}

\section{Joint workload distribution}\label{sec:JointWork}
So far we have focused on the marginal workload distribution at the first queue.
A much harder problem is to determine the steady-state joint workload distribution.
In this section, we begin the exploration of this problem, outlining
a possible approach as well as the mathematical complications arising.

Let $\tilde{v}(s_1,s_2):= \mathbb{E}({e}^{-s_1V_1(T_1+T_2) -s_2 V_2(T_1+T_2)})$
denote the steady-state joint workload LST at endings of visit periods at the second queue.
Reiterating Steps 1 - 4 of Section~\ref{sec:marginal}, but now taking both workloads into account,
leads after lengthy calculations to the following functional equation
\begin{equation}
\begin{split}
\tilde{v}(s_2,s_1) =& \frac{c_1}{c_1 -s_1 + \lambda_1(1 - \tilde{b}_1(s_1)) + \lambda_2(1 - \tilde{b}_2(s_2))}  \\ &[\tilde{v}(s_1,s_2) - \frac{s_1}{\omega_1(s_2)} \tilde{v}(\omega_1(s_2),s_2)],
~~~\mathrm{Re}~s_1,\ \mathrm{Re}~s_2 \geq 0,
\label{eq6.1}
\end{split}
\end{equation}
with
%$d_i(s_i) := \lambda_i(1 - \tilde{b}_i(s_i))$, $i=1,2$,
$\omega_1(s_2) := c_1 + \lambda_2(1 - \tilde{b}_2(s_2)) + \lambda_1(1 - \mu(\zeta,1))$;
% and $\zeta := c_1 + \lambda_2(1 - \tilde{b}_2(s_2))$
as before, $\mu(s,1)$ is the busy period LST of the M/G/1 queue in isolation corresponding to the first queue.

Let us now restrict ourselves to the fully symmetric case $c_1 = c_2 = c$, $\lambda_1 = \lambda_2 = \lambda$,
$\tilde{b}_1(s)=\tilde{b}_2(s)=\tilde{b}(s)$.
Formula (\ref{eq6.1}) then becomes
\begin{equation}
\begin{split}
\tilde{v}(s_2,s_1) =& \frac{c}{c -s_1 + \lambda(1 - \tilde{b}(s_1)) + \lambda(1 - \tilde{b}(s_2))} \\ &[\tilde{v}(s_1,s_2) - \frac{s_1}{\omega_1(s_2)} \tilde{v}(\omega_1(s_2),s_2)].
\label{eq6.2}
\end{split}
\end{equation}
Taking $s_1 = s_2$ in \eqref{eq6.2} allows us to express $\tilde{v}(\omega_1(s_2),s_2)$ in terms of $\tilde{v}(s_2,s_2)$, thus reducing (\ref{eq6.2}) to
\begin{equation}
\begin{split}
\tilde{v}(s_2,s_1) =& \frac{c}{c -s_1 +  \lambda(1 - \tilde{b}(s_1)) + \lambda(1 - \tilde{b}(s_2))}  \\ &[\tilde{v}(s_1,s_2) - \frac{s_1}{s_2}  \frac{s_2 - 2 \lambda(1 - \tilde{b}(s_2))}{c} \tilde{v}(s_2,s_2)].
\label{eq6.3}
\end{split}
\end{equation}
Interchanging all indices, one obtains a mirrored equation of (\ref{eq6.3}), and the two equations combined yield
\begin{equation}\label{eq6.4}
\begin{split}
&K(s_1,s_2) \tilde{v}(s_1,s_2) \\ &=
\frac{s_2}{s_1} \left(s_1 - 2\lambda(1 - \tilde{b}(s_1))\right)\left(c-s_1+\lambda(1 - \tilde{b}(s_1)) + \lambda(1 - \tilde{b}(s_2))\right) \tilde{v}(s_1,s_1) \\
& + \frac{s_1}{s_2} c \left(s_2 - 2\lambda(1 - \tilde{b}(s_2))\right) \tilde{v}(s_2,s_2), ~~\mathrm{Re}~s_1,~ \mathrm{Re}~s_2 \geq 0,
\end{split}
\end{equation}
with $K(s_1,s_2)=c^2 -\left(c - s_1 + \lambda(1 - \tilde{b}(s_1)) + \lambda(1 - \tilde{b}(s_2))\right)\left(c-s_2 +  \lambda(1 - \tilde{b}(s_1))\right. \\ \left.+ \lambda(1 - \tilde{b}(s_2))\right)$.
This is a so-called boundary value problem equation.
Equations of this type have been studied in the monograph \cite{cohen2000boundary}.
There an approach is outlined that, for the present problem, amounts to the following global steps:
\begin{description}
	\item[Step 1:] Consider the zeros of the {\em kernel} equation $K(s_1,s_2)$,
	that have $\mathrm{Re}~s_1,\ \mathrm{Re}~s_2 \geq 0$.
	For such pairs $(s_1,s_2)$, $\tilde{v}(s_1,s_2)$ is analytic, and hence, for those pairs, the right hand side of (\ref{eq6.4}) is equal to zero.
	\item[Step 2:] For the pairs $(s_1,s_2)$ satisfying Step 1, one needs to translate the fact that the right hand side of Equation \eqref{eq6.4} is zero into a Riemann
	or Riemann-Hilbert boundary value problem.
	The solution of such a problem yields $\tilde{v}(s_1,s_1)$ and $\tilde{v}(s_2,s_2)$. Then $\tilde{v}(s_1,s_2)$ follows via \eqref{eq6.4}.
\end{description}
Unfortunately, the above steps do not constitute a simple, straightforward recipe.
For example, several choices of zero pairs are possible in the present problem,
and it is not a priori clear what is the best choice.
A natural choice, due to the symmetry of the underlying problem, seems to be to restrict oneself to complex conjugate points, i.e., choose $(s_1,s_2) = (z,\bar{z})$.
The {\em kernel} then becomes
\[
K(z,\bar{z}) = c^2 -\left(c-z + 2\lambda \mathrm{Re}(1 - \tilde{b}(z))\right)\left(c- \bar{z} + 2\lambda \mathrm{Re}(1 - \tilde{b}(z))\right).
\]
Taking
\begin{equation}\label{Eq.contour}
c-z + 2\lambda \mathrm{Re}(1 - \tilde{b}(z)) = c {e}^{i \theta}, ~~~
c- \bar{z} + 2\lambda \mathrm{Re}(1 - \tilde{b}(z)) = c {e}^{-i \theta}, \ \theta\in[0, 2 \pi],
\end{equation}
indeed yields $K(z,\overline{z}) = K\left(z(\theta),\overline{z}(\theta)\right) = 0$, for all $\theta\in[0, 2 \pi]$,
while it is readily checked that for each such $\theta$ there is a unique $z(\theta)$ with $\mathrm{Re}~z(\theta) \geq 0$.

Turning to Step 2, one sees that the $z(\theta)$ satisfying \eqref{Eq.contour} describe a closed contour, say $L$, in the right half plane, for $\theta: 0 \rightarrow 2 \pi$,
while the fact that the right hand side of \eqref{eq6.4}, after a division by $s_1s_2 = z \bar{z}$, is zero for all these $(s_1,s_2) = (z(\theta), \overline{z}(\theta))$
translates into the following relation
\begin{equation}
\mathrm{Re} \left[\frac{1}{z} \left(1 - \rho \tilde{b}^r(z))\right) \tilde{v}(z,z) {e}^{\frac{1}{2} i \theta}\right] = 0,\ z \in L,
\label{eq6.5}
\end{equation}
with $\tilde{v}(z,z)$ and  $\tilde{b}^r(z) = \frac{1 - \tilde{b}(z)}{z \mathbb{E} B}$ analytic inside $L$.
The expression inside the square brackets of (\ref{eq6.5}) is analytic inside $L$ apart from the pole at $z=0$.
This results in a similar boundary value problem as has been treated in \cite{BoxmaVanHoutum}.
%or would have been analytic except for a pole, then
%we would have obtained a Riemann-Hilbert problem on contour $L$, see, e.g., \cite[Chapters~II~and~IV]{Gakhov}
%or \cite[Section~I.3]{cohen2000boundary}.
The solution of such a problem is known when $L$ is the unit circle. For other closed contours, one needs a conformal mapping of
that contour onto the unit circle;
several procedures are available for obtaining such conformal mappings.
%Of course $\bar{z}$ is {\em not} analytic, so we have not yet arrived at a standard Riemann-Hilbert boundary value problem.
%Just like with  the Wiener-Hopf technique in the related Wiener-Hopf boundary value problem, there might be a way around this
%by applying a suitable  (Wiener-Hopf) factorization;
%this is a path we would like to explore in future research.
\begin{remark}
In a future study we aim to handle all the technicalities which arise in treating this boundary value problem with a pole.
	%There are several open problems emerging at this point.
	When we manage to solve the present symmetric problem, we are still faced with the more general asymmetric two-queue problem.
	Subsequently, one could turn to the joint {\em queue length} distribution.
	However, a complication there is that a switch between queues might occur during a service time,
	forcing one to keep track of the length of the residual service time.
	From that perspective, workload seems to be an easier quantity than queue length.
\end{remark}

In the next section, we turn our focus to the joint queue length distribution, but restricting ourselves to exponential service time distributions,
so we do not need to keep track of the residual service time.
Instead of pursuing a boundary value approach, we explore a perturbation approach, which allows us to
derive an analytic expansion for the joint queue length distribution.

%=================== Section 5: Joint queue length analysis =========================%

\section{Joint queue length distribution}\label{sec:JointQueue}
In this section, we turn our attention to the steady-state joint queue length distribution, restricting ourselves to {\em exponential  service requirement distributions} in both queues, with rates $\mu_i=1/\mathbb{E}(B_i)$, $i=1,2$, respectively. Under this assumption, we do not need to keep track of the residual service times, which simplifies the analysis. However, a direct analytic derivation of the joint queue length distribution (or its PGF) turns out to be as challenging as the analysis presented in Section \ref{sec:JointWork}.
To address this issue, in this section, we explore the use of parametric perturbation for the derivation of the joint queue length distribution.
In what follows, we use the framework developed in Altman et al. \cite{altman2004perturbation};
we perturb the service and arrival rates by a common parameter, denoted by $\varepsilon\geq0$, i.e., the perturbed service rate of the customers in queue $i$ is $\varepsilon \mu_i$, $i=1,2$, and arrivals occur according to two independent Poisson processes with perturbed rates $\varepsilon \lambda_i$, $i = 1, 2$.  The parameters that are not perturbed are $c_i, i=1,2$ i.e., the rates of the exponentially distributed durations that the server spends in each queue.
Note that the stability condition \eqref{Stability} is not affected by this scaling.

\noindent
The perturbed process is a continuous time Markov chain defined on the state space
\begin{equation*}
\mathcal{S} = \big\{(n_1, n_2, k),\ n_1, n_2 \in \mathbb{N},\ k \in \{1, 2\} \big\},
\end{equation*}
in which $n_i$ denotes the queue length in queue $i$, $i=1,2$, and the third element in the state space description reports the queue in which the server is active. Furthermore, let $\bm{G}(\varepsilon)$ denote the generator of the perturbed Markov process. We decompose this perturbed generator into the unperturbed generator $\bm{G}^{(0)}$ and the perturbation matrix $\bm{G}^{(1)}$,
\begin{equation}\label{MCPG}
\bm{G}(\varepsilon) = \bm{G}^{(0)} + \varepsilon \bm{G}^{(1)},
\end{equation}
so as to investigate the dependence of the stationary joint queue length distribution on the parameter $\varepsilon$.
%For the description of the generator, we use a lexicographic ordering on the states $(n_1,n_2,k)$.
The unperturbed generator $\bm{G}^{(0)}$ corresponds to the states depicting a change of the state of the server from one queue to the other; it is given by
\begin{eqnarray}\label{Eq:G0_matrix}
\bm{G}^{(0)} = \begin{bmatrix}
\bm{C} & \mathbf{0}_{2 \times 2} & \cdots  \\
\mathbf{0}_{2 \times 2} & \bm{C}  & \cdots  \\
\vdots & \vdots  & \ddots  \\
\end{bmatrix},		
\end{eqnarray}
with $\bm{C} =  \begin{bmatrix} -c_1 & c_1 \\
c_2 & -c_2 \\
\end{bmatrix}$, and $\bm{0}_{2\times2}$ a $2\times2$ matrix of zeros. Throughout the remainder of the paper we use this notation with subscripts to indicate the dimension when needed.
When the dimension is clear from the context, the index is omitted; note that the dimension can be infinite. 	\\
The perturbation matrix $\bm{G}^{(1)}$ is defined in terms of its elements, with $n_1\geq0$, $n_2\geq0$, $k=1,2$,
\begin{equation}
\begin{split}
&\bm{G}^{(1)}_{(n_1, n_2, k), (n_1 + 1, n_2, k)} = \lambda_1, ~ \bm{G}^{(1)}_{(n_1, n_2, k), (n_1, n_2 + 1, k)} = \lambda_2, \\&
\bm{G}^{(1)}_{(n_1+1, n_2, 1), (n_1 , n_2, 1)} = \mu_1,~   \bm{G}^{(1)}_{(n_1, n_2+1, 2), (n_1, n_2, 2)} = \mu_2 ,\\&
\bm{G}^{(1)}_{(n_1, n_2, k), (n_1, n_2, k)} = -\left(\lambda_1 + \lambda_2 + \mu_k \one_{\{n_k \geq 1\}}\right),
\label{Eq:G1_matrix_first_queue}
\end{split}
\end{equation}
with $\one_{\{n_k \geq 1\}}$ an indicator function taking value 1, if $n_k \geq 1$, and 0, otherwise.

\noindent
In order to implement the framework of Altman et al.~\cite{altman2004perturbation}, it is convenient to first define the transition probability matrix $\bm{P}(\varepsilon) = \bm{I}+\Delta \bm{G}(\varepsilon)$ of the corresponding (uniformized) discrete time perturbed Markov chain ($\bm{I}$ being the identity matrix).
In order to simplify notation, in what follows, we assume without loss of generality that
\begin{equation}\label{uniformization}
\lambda_1 + \lambda_2 + \mu_1 + c_1 \leq 1 ~~ \mbox{and} ~~ \lambda_1 + \lambda_2 + \mu_2  + c_2 \leq 1.
\end{equation}
Note that indeed, this assumption simply entails a scaling of time.
Still, it allows us to take $\Delta=1$ and it ensures that
\begin{equation}\label{Pepsilon}
\bm{P}(\varepsilon) = \bm{I}+\bm{G}(\varepsilon),
\end{equation}
is a probability matrix for all $\varepsilon\in[0,1]$, which is convenient.
We remind the reader that our ultimate goal is to find (or approximate) the stationary measure belonging to $\bm{G}(1)$ (and, equivalently, of the discrete time counter part $\bm{P}(1)$).
In order to achieve that, we  first establish the analyticity of the stationary distribution for $\varepsilon$ in a punctured neighborhood of 0, cf.~Theorem~\ref{Theo:IPM} below.
We emphasize that it is not guaranteed that the stationary distribution will be analytic up to $\varepsilon=1$.
The analysis in~\cite{altman2004perturbation} gives a lower bound for the radius of convergence, which in general turns out to be rather conservative.
\\
Note that the perturbed transition probability matrix $\bm{P}(\varepsilon)$ can also be decomposed into the unperturbed probability matrix $\bm{P}^{(0)}$ and the perturbation matrix $\bm{G}^{(1)}$, with $\bm{P}^{(0)}=\bm{I}+\bm{G}^{(0)}$, i.e.,
\begin{align}\label{matrixP0}
\bm{P}^{(0)} = \begin{bmatrix}
\bm{I}_{2 \times 2}  + \bm{C} & \mathbf{0}_{2 \times 2} & \cdots  \\
\mathbf{0}_{2 \times 2} & \bm{I}_{2 \times 2}  + \bm{C}  & \cdots  \\
\vdots & \vdots  & \ddots  \\
\end{bmatrix}.
\end{align}
It is evident that the unperturbed process consists of several ergodic classes, making our setting fit the singular perturbation approach in~\cite{altman2004perturbation}.

\subsection{Singular perturbation analysis: outline}
Following the analysis performed in \cite{altman2004perturbation}, we now formulate four conditions based on which the invariant probability measure of the perturbed Markov chain, denoted by $\pi(\varepsilon)$, is derived. These four conditions guarantee the analyticity of $\pi(\varepsilon)$ in $\varepsilon$ in a punctured neighborhood of zero. Furthermore, they guarantee that the coefficients of the power series $\pi(\varepsilon)=\sum_{m = 0}^{\infty} \varepsilon^m \pi^{(m)}$ form a geometric sequence and, hence, that there exists a computationally stable updating formula for $\pi(\varepsilon)$, see \cite{altman2004perturbation}.
\\
In this subsection we only formulate the four conditions and give the main result of the section. The detailed mathematical proofs follow in the next subsection.

\begin{asmpt}\label{Assumption1}
	The unperturbed Markov chain consists of several (denumerable) ergodic classes and there are no transient states.
\end{asmpt}

\noindent
There is an ergodic class for each $\bm{i}\in \big\{(n_1, n_2),\ n_1, n_2 \in \mathbb{N}\big\}$, i.e., in an ergodic class, the numbers of customers in both queues are fixed.
All ergodic classes are identical, and consist of two states, $k\in\{1,2\}$, indicating the queue being served.

\begin{asmpt}\label{Assumption2}
	The Markov chains corresponding to the ergodic classes of the unperturbed Markov chain are \textit{uniformly} Lyapunov stable i.e., for each ergodic class there exist a strongly aperiodic state $\alpha\in\{1,2\}$ (with a strictly positive probability on the corresponding diagonal element of the transition matrix $\bm{I}+\bm{C}$, with the matrix $\bm{C}$ given in \eqref{Eq:G0_matrix}),  constants $0<\delta < 1$ and $b < \infty$, and a Lyapunov function $\bm{u} = (\begin{array}{c c} u_1 & u_2 \end{array})'$, with $u_i \geq 1$, $i = 1, 2$, such that
	\begin{equation} \label{AS2}
	(\bm{I} + \bm{C}) \bm{u} \leq \delta \bm{u} + b \bm{e}_{\alpha},
	\end{equation}
	with $\bm{e}_{\alpha}$ a vector with 1 in the entry belonging to state $\alpha$ and zero in the other entry.
\end{asmpt}
For the next condition, we first introduce the {\em aggregated} Markov chain \cite{courtois2014decomposability, delebecque1983reduction, pervozvanskii2013theory}, with generator $\bm{\Gamma}$, given in matrix form as follows
\begin{equation}\label{Eq:Genrator_aggregated_MC}
\bm{\Gamma} = \bm{V} \bm{G}^{(1)} \bm{W},
\end{equation}
with $\bm{V}$ and $\bm{W}$ defined as in \cite[p.~844]{altman2004perturbation}; $\bm{V}$ (resp., $\bm{W}$) is a matrix whose rows (columns) correspond to the ergodic classes and its columns (rows) to the states in $\mathcal{S}$.
The $\bm{i}$-th row of $\bm{V}$ is the invariant measure of the unperturbed Markov chain, given that the process starts in the $\bm{i}$-th ergodic class, with $\bm{i}\in \big\{(n_1, n_2),\ n_1, n_2 \in \mathbb{N}\big\}$, i.e.,
\begin{eqnarray}\label{Eq:V_matrix}
\bm{V} = \begin{bmatrix}
\bm{\tilde{C}} & 	\bm{0}_{1\times 2}  & \bm{0}_{1\times 2} &\cdots \\
\bm{0}_{1\times 2} & \bm{\tilde{C}}  & \bm{0}_{1\times 2} &\cdots\\
\bm{0}_{1\times 2} & \bm{0}_{1\times 2} & \bm{\tilde{C}}   &\cdots \\
\vdots & \vdots & \vdots & \ddots
\end{bmatrix},
\end{eqnarray}	
with $\bm{\tilde{C}} =  \begin{bmatrix} c_2/(c_1+c_2) & c_1/(c_1+c_2)\\
\end{bmatrix}$.
The $\bm{j}$-th column of $\bm{W}$ has ones in the components corresponding to the $\bm{j}$-th ergodic class and zeros in the other components, with $\bm{j}\in \big\{(n_1, n_2),\ n_1, n_2 \in \mathbb{N}\big\}$, i.e.,
\begin{eqnarray}\label{Eq:W_matrix}
\bm{W} &=& \begin{bmatrix}
\bm{1}_{2\times1}&\bm{0}_{2\times1}&\bm{0}_{2\times1}&\cdots\\
\bm{0}_{2\times1}&\bm{1}_{2\times1}&\bm{0}_{2\times1}&\cdots\\
\bm{0}_{2\times1}&\bm{0}_{2\times1}&\bm{1}_{2\times1}&&\cdots\\
\vdots & \vdots & \vdots & \ddots
\end{bmatrix}
\end{eqnarray}
with $\bm{1_{2\times1}}=\begin{bmatrix}1\\1\end{bmatrix}$.\\
Hence, for $n_1\geq0$, $n_2\geq0$, the elements of the generator matrix $\bm{\Gamma}$ are:
\begin{equation}
\begin{split}
&\bm{\Gamma}_{(n_1, n_2), (n_1 + 1, n_2)} = \lambda_1 , ~ \bm{\Gamma}_{(n_1, n_2), (n_1, n_2 + 1)} = \lambda_2,  \\ &\bm{\Gamma}_{(n_1 + 1, n_2), (n_1, n_2)} = \mu_1 \frac{c_2}{c_1 +c_2}, ~
\bm{\Gamma}_{(n_1, n_2 + 1), (n_1, n_2)} = \mu_2 \frac{c_1}{c_1 +c_2},  \\  &\bm{\Gamma}_{(n_1, n_2), (n_1, n_2)} = -\left(\lambda_1 + \lambda_2 + \mu_1 \frac{c_2}{c_1 +c_2} \one_{\{n_1 \geq 1\}} + \mu_2 \frac{c_1}{c_1 +c_2} \one_{\{n_2 \geq 1\}}\right).\label{GGamma}
\end{split}
\end{equation}	
It is convenient to think of the aggregated Markov chain as the limiting joint queue length process as $\varepsilon\to0$. In this limit, the server moves infinitely fast between the two queues, making them two independent M/M/1 queues with arrival rates $\lambda_i$ and service rates ${\mu_i}\frac{c_1c_2/{c_i}}{c_1+c_2}$, $i=1,2$. Based on this remark, one can immediately deduce that the invariant probability measure of the aggregated Markov chain is
\begin{align}\label{IPMAggr}
\bar{\pi}(n_1,n_2)=(1-\tilde{\rho}_1)\tilde{\rho}_1^{n_1}(1-\tilde{\rho}_2)\tilde{\rho}_2^{n_2},\ n_1,n_2\geq0,
\end{align}
with $\tilde{\rho}_i=\frac{\lambda_ic_i(c_1+c_2)}{\mu_ic_1c_2}$, $i=1,2$.
\\
We are now ready to state the third condition.
\begin{asmpt}\label{Assumption3}
	The aggregated Markov chain is irreducible and Lyapunov stable, i.e., there exist a strongly aperiodic state $\bar{\alpha}=(n_1,n_2)$ (with a strictly positive probability on the diagonal of the transition matrix $\bm{I}+ \bm{\Gamma}$, with the matrix $\bm{\Gamma}$ given in \eqref{GGamma}), constants $0 < \overline{\delta} < 1$, $\overline{b} < \infty$ and a Lyapunov function $\overline{\bm{u}}=\left(\overline{u}_{(n_1, n_2)}\right)_{(n_1, n_2)\in\mathbb{N}^2}$, with elements $\overline{u}_{(n_1, n_2)} \geq 1$, for all $n_1, n_2\geq0$, such that
	\begin{equation}\label{EqAssumption3}
	(\bm{I} + \bm{\Gamma}) \overline{\bm{u}} \leq \bar{\delta}\overline{\bm{u}} + \bar{b} \bm{e}_{\overline{\alpha}}.
	\end{equation}		
\end{asmpt}

\begin{asmpt}\label{Assumption4}
	The perturbation matrix $\bm{G}^{(1)}$ is $\tilde{\bm{u}}$-bounded (for $\tilde{u}_{\bm{i}k}=\overline{u}_{\bm{i}}u_k$, with $\bm{i}\in \big\{(n_1, n_2),\ n_1, n_2 \in \mathbb{N}\big\}$ and $k=1,2$) or, equivalently,
	\begin{align}
	\parallel \bm{G}^{(1)} \parallel_{\tilde{\bm{u}}} :=\sup_{\bm{s}\in\mathcal{S}}\tilde{u}_{\bm{s}}\,{}^{-1}\sum_{\overline{\bm{s}}\in\mathcal{S}} \left | \bm{G}^{(1)}_{\bm{s},\overline{\bm{s}}} \right |\tilde{u}_{\overline{\bm{s}}}
	\label{u-Bounded}
	\end{align}
	is bounded by some constant $g>0$, cf. \cite[p.~841]{altman2004perturbation}.
\end{asmpt}
Note that, because of the repetitive structure of $\bm{G}^{(0)}$, this assumption implies that $\bm{P}(\varepsilon)$ is $\tilde{\bm{u}}$-bounded for all $\varepsilon\geq0$.
\\

\noindent
We can now state the main theorem of the section, which is based on~\cite[p.~845, Theorem~4.1]{altman2004perturbation}.

%================================== Main Theorem =============================
\begin{theorem}\label{Theo:IPM}
	Under Conditions \ref{Assumption1}--\ref{Assumption4}, the perturbed Markov chain has a unique invariant probability measure, $\pi(\varepsilon)$, which is an analytic function of $\varepsilon$ in a neighborhood of 0,
	\begin{equation} \label{IPM1}
	\pi(\varepsilon) = \sum_{m = 0}^{\infty} \varepsilon^m \pi^{(m)}, ~~ \pi^{(m)} = \bar{\pi} \, \bm{V} \, \bm{U}^m,
	\end{equation}
	where $\bar{\pi}$ is the invariant probability measure of the aggregated Markov chain, cf. \eqref{IPMAggr}, and 			 	
	\begin{equation} \label{IPM2}
	\bm{U} = \bm{G}^{(1)} \bm{H} \left(\bm{I} + \bm{G}^{(1)}\, \bm{W} \,\bm{\Phi}\, \bm{V} \right),
	\end{equation}
	$\bm{V}$ and $\bm{W}$ are given in \eqref{Eq:V_matrix} and \eqref{Eq:W_matrix}, respectively, and $\bm{H}$ and $\bm{\Phi}$ the deviation matrices of the unperturbed and aggregated Markov chains, respectively, are given by
	\begin{equation}\label{Eq:H_Matrix}
	\bm{H} =  - \frac{1}{{(c_1 + c_2)}^2} \bm{G}^{(0)},
	\end{equation}	
	and
	\begin{equation}\label{Eq:Deviation_matrix_AggMC}
	\bm{\Phi} = \sum_{m = 0}^{\infty} \left[(\bm{I} + \bm{\Gamma})^m - \bm{\gamma}\right].
	\end{equation}
	Here $\bm{\gamma}$ is the ergodic projection of the aggregated Markov chain, with generator $\bm{\Gamma}$ given in \eqref{GGamma}, i.e.,
	\begin{equation*}
	\bm{\gamma}=\lim_{n\to\infty}\frac{1}{n}\sum_{m=1}^n (\bm{I}+\bm{\Gamma})^m.
	\end{equation*}
\end{theorem}

\begin{remark}
	We do not discuss the radius of convergence of the series in~\eqref{IPM1}.
	Theorem~4.1 of~\cite{altman2004perturbation}  gives a (rather conservative) bound for the analyticity region.
\end{remark}

\begin{remark}
	The invariant probability measure of the perturbed Markov chain can be calculated by the updating formula
	\begin{equation}
	\pi(\varepsilon)=\pi^{(0)}\left(\bm{I}-\varepsilon \bm{U}\right)^{-1},
	%\ 0<\varepsilon\leq\overline{\varepsilon},
	\end{equation}
	with  $\varepsilon$ in a neighborhood of 0, cf. \cite[p.~845, Theorem~4.1]{altman2004perturbation}.
\end{remark}

\begin{remark}
	In order to calculate the deviation matrix $\bm{\Phi}$, one may use the following equations
	\begin{align*}
	\bm{\Phi}\, \bm{\Gamma} &= \bm{\Gamma} \,\bm{\Phi} = \bm{\gamma} - \bm{I},\\
	\bm{\gamma}\, \bm{\Phi} &= \bm{\Phi}\, \bm{\gamma} = \bm{0}.
	\end{align*}
	We briefly describe two approaches to obtain the deviation matrix $\bm{\Phi}$: an analytic one involving PGFs and a numerical one. Both approaches require some additional work. The analytic approach, which involves the consideration of generating functions, leads to a boundary value problem for which we can employ Steps 1 and 2 discussed in Section \ref{sec:JointWork}. Performing these steps reveals a problem similar to the combinatorial random walk in the quadrant with transitions to the West, North, and South-East, cf. \cite[Section~5.2]{Bousquet-Melou}. In order to obtain the expression for $\bm{\Phi}$, we need to invert the obtained PGF.
	A numerical approach is to truncate the state space and solve numerically the corresponding finite system of equations above. We do remark that truncating the state space is a delicate task, since the entries of $\bm{\Phi}$ corresponding to states far from the origin are unbounded.
We do not further investigate this in this paper.
\end{remark}

\subsection{Singular perturbation analysis: verification of the conditions}

It remains to prove that Conditions \ref{Assumption1} - \ref{Assumption4} are satisfied and also to indicate how the deviation matrix of the unperturbed Markov chain, $\bm{H}$, is calculated.

\paragraph{Verification of Condition \ref{Assumption1}}	
As explained in the previous section, this condition follows directly from Equation \eqref{matrixP0}.

\paragraph{Verification of Condition \ref{Assumption2}}
Obviously, all ergodic classes are identical and contain two states ($k=1,2$), thus this condition is trivially met, but for the construction in the remainder it is useful to specify the Lyapunov function used.

First note that the strong aperiodicity follows from the uniformization condition \eqref{uniformization}. We can choose any of the two states as the strongly aperiodic state; in the following we choose $\alpha:=1$. To construct the Lyapunov function first we choose the constants $\delta$  and $b$ as $\delta \in \left(1 - c_2, 1 - \frac{c_1c_2}{c_1 + c_2}\right)$,  $b = 1 - \delta + \frac{c_1^2}{c_2} $. Then we can verify that the Lyapunov function
\begin{equation}\label{Eq:Nu}
\bm{u} =\begin{bmatrix} 1 \\ 1+\frac{c_1}{c_2} \end{bmatrix}
\end{equation}
satisfies  \eqref{AS2}. It also follows that, indeed, $\delta \in (0, 1)$, $0 < b < \infty$ and $u_k \geq 1$, $k= 1, 2$.

\paragraph{Verification of Condition \ref{Assumption3}}
From the definition of the generator of the aggregated Markov chain, cf. \eqref{GGamma}, and the stability condition \eqref{Stability}, it is immediately evident that the aggregated Markov chain is ergodic, since it behaves as two independent ergodic M/M/1 queues with arrival rate $\lambda_i$ and service rate $\frac{\mu_i}{c_i}\frac{c_1c_2}{c_1+c_2}$, $i=1,2$.

Now by using the uniformization condition \eqref{uniformization}, state $(0, 0)$ is strongly aperiodic i.e., we may choose $\bar{\alpha} = (0, 0)$. We proceed to describe the Lyapunov function $\overline{\bm{u}}$ and the constants $\bar{\delta} \in (0, 1)$ and $\bar{b}$ which satisfy Condition \ref{Assumption3}. Note that Equation \eqref{EqAssumption3} is written as follows, for $n_1, n_2 \geq 0$,
\begin{align*}
& \left(1 - \left(\lambda_1 + \lambda_2 + \mu_1 \frac{c_2}{c_1 +c_2}\one_{\{n_1\geq1\}} + \mu_2 \frac{c_1}{c_1 +c_2} \one_{\{n_2\geq1\}}\right) \right)\overline{u}_{(n_1, n_2)} \nonumber \\
&+\lambda_1 \overline{u}_{(n_1 + 1, n_2)} + \lambda_2 \overline{u}_{(n_1, n_2 + 1)} + \mu_1 \frac{c_2}{c_1 +c_2}  \one_{\{n_1\geq1\}} \overline{u}_{(n_1 - 1, n_2)}
 \nonumber\\ &+ \mu_2 \frac{c_1}{c_1 +c_2} \one_{\{n_2\geq1\}}  \overline{u}_{(n_1, n_2-1)}
\leq \bar{\delta}  \overline{u}_{(n_1, n_2)}+ \bar{b}\one_{\{(n_1,n_2)=(0,0)\}} .%\label{AS3_2}
\end{align*}
Solving the above equations with equality, after choosing
\begin{equation}\label{Eq:Nu_bar}
\overline{u}_{(n_1, n_2)} = {\left(\sqrt{\frac{\mu_1 c_2}{\lambda_1(c_1 + c_2)}}\right)}^{n_1} {\left(\sqrt{\frac{\mu_2 c_1}{\lambda_2(c_1 + c_2)}}\right)}^{n_2},
\end{equation}
yields the solution for $\bar{\delta}$ and $\bar{b}$. We choose
\begin{align*}%\label{AS3_8}
\bar{\delta}  =& 1 - {\left(\sqrt{\lambda_1} - \sqrt{\mu_1 \frac{c_2}{c_1 +c_2}}\right)}^{2} - {\left(\sqrt{\lambda_2} - \sqrt{\mu_2 \frac{c_1}{c_1 +c_2}}\right)}^{2} \nonumber\\
&+ \max{\left\{\mu_2 \frac{c_1}{c_1 +c_2} \left(1- \sqrt{\frac{\lambda_2 (c_1 + c_2)}{\mu_2 c_1}}\right), \mu_1 \frac{c_2}{c_1 +c_2} \left(1- \sqrt{\frac{\lambda_1 (c_1 + c_2)}{\mu_1 c_2}}\right)\right\}},
\end{align*}	
and
\begin{align*}
\bar{b} =  1 - \bar{\delta} + \lambda_1 \left(\sqrt{\frac{\mu_1 c_2}{\lambda_1(c_1 + c_2)}} - 1\right) + \lambda_2\left(\sqrt{\frac{\mu_2 c_1}{\lambda_2(c_1 + c_2)}} - 1\right).
\end{align*}
Note that due to the uniformization condition \eqref{uniformization}, indeed $\bar{\delta} \in (0, 1)$, $0 < \bar{b} < \infty$ and $\overline{u}_{(n_1, n_2)} \geq 1,$ for all $n_1, n_2\geq0.$

\paragraph{Verification of condition \ref{Assumption4}}
To verify this assumption, we apply the definition, cf. \eqref{u-Bounded}, and show that
$$\parallel \bm{G}^{(1)} \parallel_{\tilde{\bm{u}}} \leq \max\{g_1, g_2\},$$
with $g_1 = \left(\frac{\mu_1 c_2}{\lambda_1(c_1 + c_2)}\right)^{-\frac{1}{2}}\left(\mu_1 + \frac{\mu_1 c_2 + \mu_2 c_1}{c_1 + c_2}\right)$ 	and  $g_2 = \left(\frac{\mu_2 c_1}{\lambda_2(c_1 + c_2)}\right)^{-\frac{1}{2}}\left(\mu_2 + \frac{\mu_1 c_2 + \mu_2 c_1}{c_1 + c_2}\right)$.\\
In order to do so, we use the following $\tilde{\bm{u}}$-norm
\[\tilde{u}_{(n_1,n_2,k)}=\overline{u}_{(n_1,n_2)}u_k,\ (n_1,n_2,k)\in\mathcal{S},\]
with $\overline{u}_{(n_1,n_2)}$ given in \eqref{Eq:Nu_bar} and $u_k$ given in \eqref{Eq:Nu}.

\paragraph{Derivation of the deviation matrix of the unperturbed Markov chain}	
It follows from Condition \ref{Assumption1}, that the deviation matrix of the unperturbed Markov chain, $\bm{H}$, has the following block diagonal structure
\begin{equation}\label{H0}
\bm{H} = \begin{bmatrix}
\bm{H}_{2 \times 2} &  \mathbf{0}_{2 \times 2}  & \cdots \\
\mathbf{0}_{2 \times 2} & \bm{H}_{2 \times 2} & \cdots \\
\vdots & \vdots  & \ddots
\end{bmatrix},
\end{equation}
with $\bm{H}_{2 \times 2}$ the deviation matrix of each ergodic class of the unperturbed Markov chain, i.e., \begin{equation}\label{H1}
\bm{H}_{2 \times 2} =  \sum_{j = 0}^{\infty} \left[{(\bm{I} +  \bm{C})}^j - \bm{c}\right],
\end{equation}
with $\bm{C}$ given in \eqref{Eq:G0_matrix} and $\bm{c}$ the ergodic projection of the unperturbed Markov chain given as
\begin{equation*}%\label{H2}
\bm{c} = \begin{bmatrix}
\frac{c_2}{c_1 + c_2} & \frac{c_1}{c_1 + c_2} \\
\frac{c_2}{c_1 + c_2} & \frac{c_1}{c_1 + c_2}       	
\end{bmatrix},
\end{equation*}
cf. \cite[p.~64,~Equation~4.1]{spieksma1990geometrically}.	\\
We evaluate \eqref{H1} using the spectral decomposition (eigen-decomposition)
of matrices $\bm{I} + \bm{C}$ and $\bm{c}$; the diagonal matrices containing the eigenvalues are
$\bm{D}_{\bm{I}+\bm{C}}=\mathrm{diag}\{1,1 - (c_1 + c_2) \}=\begin{bmatrix}
1 & 0 \\
0 & 1 - (c_1 + c_2)      	
\end{bmatrix}$
and
$\bm{D}_{\bm{c}}=\mathrm{diag}\{1,0 \}= \begin{bmatrix} 	1 & 0 \\
0 & 0  	\end{bmatrix}$, respectively, and the corresponding matrix of eigenvectors is $\bm{M} = \begin{bmatrix} 1 & -c_1 \\
1 & c_2  \end{bmatrix}$. Naturally, in dimension 2, both matrices produce the same eigenvectors because $\bm{c}$ is the ergodic projection of $\bm{I} + \bm{C}$.
Therefore, Equation \eqref{H1} can be written as
\begin{align}
\bm{H}_{2 \times 2} &= \bm{M} \left(\sum_{m = 0}^{\infty} \left[ \bm{D}_{\bm{I}+\bm{C}}^m - \bm{D}_{\bm{c}} \right] \right) \bm{M}^{-1}\nonumber\\
&=    - \frac{1}{{(c_1 + c_2)}^2}\bm{C}.\label{H4}
\end{align}
Combining \eqref{H4} and \eqref{H0} yields Equation \eqref{Eq:H_Matrix}.

%=================== Section 6: Conclusion and future possibilities ================%	

\section{Possible future directions}\label{sec:FutureDir}
We have studied a single server two-queue polling model with a random residing time service discipline. More concretely, we considered that customers arrive at the two queues according to two independent Poisson processes. There is a single server that serves both queues with generally distributed service times. The server spends an exponentially distributed amount of time in each queue. After the completion of this residing time, the server instantaneously switches to the other queue, i.e., there is no switch-over time. A service discipline with a random residing time does not satisfy the so-called branching property \cite{Resing}, which significantly complicates the underlying analysis.

For this polling model, we derived the steady-state marginal workload distribution and
used it to obtain several asymptotic results.
We also discussed the complications arising in the calculation of the joint workload distribution. Furthermore, restricting ourselves to the case of exponential service times, we have calculated the joint queue length distribution using (singular) perturbation analysis. 
It is a topic for further research to determine how to efficiently truncate the
system without inducing too large errors.
The insights gained for the two-queue polling model, specifically for the derivation of the marginal workload, cf. Section \ref{sec:marginal}, can be also used in the case of $N$ queues, $N>2$. In addition, one may generalize the compound Poisson input to a L\'evy subordinator input process.
%Also, the analysis at hand stands in the case of dependent arrivals streams at the queues.

Another interesting topic for future research is to develop the framework for the derivation of the bivariate LST of the joint workload distribution,
in particular in the asymmetric case, cf. Section \ref{sec:JointWork}, and for the derivation of the bivariate PGF of the joint queue length distribution in the case of exponential service requirements.
In particular, the objective in such a setting is to develop an approach for the transformation of Equation \eqref{eq6.4},
and its version for the asymmetric case,
into a Riemann or Riemann-Hilbert boundary value problem. This requires, that we first choose the zeros of the kernel equation $K(s_1,s_2)$, so as to define a closed smooth contour. Thereafter, we need to show that Equation \eqref{eq6.4} on the contour reduces to the study of an analytic function
(probably with the exception of one pole) with a known boundary condition.
%The main challenge of such an approach lies in the fact that the typical choice of complex conjugate points does not reveal an analytic function, cf. Equation \eqref{eq6.5}, thus indicating that we may need to apply a different approach.
An interesting alternative direction would be to extend the framework developed by Fayolle et al. \cite{fayolle1999}, of the systematic use of the {\em kernel method} using the group of {\em birational transformations} that leave the kernel equation unchanged. The challenge in our case is that the kernel $K(s_1,s_2)$ does not have the regular structure indicated in \cite{fayolle1999}, but this does not seem to impose an insuperable obstacle, see, e.g., \cite{van2008partially}.

\section*{Acknowledgments}
The authors gratefully acknowledge useful discussions with Ahmad Al Hanbali, Offer Kella, Socrates Olympios, Shelly Zacks and Bert Zwart.
The research of Mayank Saxena was funded by the NWO TOP-C1 project of the Netherlands Organisation for Scientific Research.
The research of Onno Boxma was partly done in the framework of the IAP BESTCOM project, funded by the Belgian government;
it was also funded by the NWO Gravitation Program NETWORKS of the Dutch government.
The work of Stella Kapodistria is supported by the NWO Gravitation Program NETWORKS of the Dutch government.

\end{document}